\documentclass[11pt, reqno, oneside]{amsart}
\usepackage{amsmath,amsfonts,amssymb,amsthm}

\usepackage{enumerate}
\usepackage{esint}
\usepackage[pagebackref,colorlinks,citecolor=blue,linkcolor=blue]{hyperref}
\usepackage[dvipsnames]{xcolor}
\usepackage{mathtools}  

\usepackage{mathtools}
\usepackage{setspace}

\usepackage{comment}

\usepackage{geometry}
\numberwithin{equation}{section}

\theoremstyle{plain}
\newtheorem{theorem}[equation]{Theorem}
\newtheorem{corollary}[equation]{Corollary}
\newtheorem{lemma}[equation]{Lemma}

\theoremstyle{definition}
\newtheorem{definition}[equation]{Definition}

\newtheorem{example}[equation]{Example}
\newtheorem{remark}[equation]{Remark}

\newcommand{\R}{{\mathbb R}}
\newcommand{\N}{{\mathbb N}}

\newcommand{\Om}{\Omega}

\providecommand{\vint}[1]{\mathchoice
	{\mathop{\vrule width 5pt height 3 pt depth -2.5pt
			\kern -9pt \kern 1pt\intop}\nolimits_{\kern -5pt{#1}}}
	{\mathop{\vrule width 5pt height 3 pt depth -2.6pt
			\kern -6pt \intop}\nolimits_{\kern -3pt{#1}}}
	{\mathop{\vrule width 5pt height 3 pt depth -2.6pt
			\kern -6pt \intop}\nolimits_{\kern -3pt{#1}}}
	{\mathop{\vrule width 5pt height 3 pt depth -2.6pt
			\kern -6pt \intop}\nolimits_{\kern -3pt{#1}}}}

\newcommand{\eps}{\varepsilon}
\newcommand{\loc}{\mathrm{loc}}

\newcommand{\BV}{\mathrm{BV}}

\newcommand{\ch}{\text{\raise 1.3pt \hbox{$\chi$}\kern-0.2pt}}

\newcommand{\mres}{\mathbin{\vrule height 2ex depth 2.2pt width
		0.12ex\vrule height -0.3ex depth 2.2pt width .5ex}}

\DeclareMathOperator{\Mod}{Mod}
\DeclareMathOperator{\AM}{AM}

\DeclareMathOperator{\dist}{dist}
\DeclareMathOperator{\diam}{diam}

\DeclareMathOperator{\Lip}{Lip}

\DeclareMathOperator{\lip}{lip}

\begin{document}
\title[]{Quasiconformal, Lipschitz, and BV mappings\\
	 in metric spaces}
\author{Panu Lahti}
\address{Panu Lahti,  Academy of Mathematics and Systems Science, Chinese Academy of Sciences,
	Beijing 100190, PR China, {\tt panulahti@amss.ac.cn}}

\subjclass[2020]{30L10, 46E36, 26B30}
\keywords{Quasiconformal mapping, Newton-Sobolev mapping,
	function of bounded variation, generalized distortion number, generalized Lipschitz number}

\begin{abstract}
Consider a mapping $f\colon X\to Y$ between two metric measure spaces.
We study generalized versions of the local Lipschitz number $\Lip f$,
as well as of the distortion number $H_f$ that is used to define quasiconformal mappings.
Using these, we give sufficient conditions for $f$ being a BV mapping $f\in\BV_{\loc}(X;Y)$
or a Newton-Sobolev mapping $f\in N_{\loc}^{1,p}(X;Y)$, with $1\le p<\infty$.
\end{abstract}

\date{\today}
\maketitle

\section{Introduction}

Consider two metric measure spaces $(X,d,\mu)$ and $(Y,d_Y,\nu)$, and a mapping $f\colon X\to Y$.
For every $x\in X$ and $r>0$, one defines
\[
L_f(x,r):=\sup\{d_Y(f(y),f(x))\colon d(y,x)\le r\}
\]
and
\[
l_f(x,r):=\inf\{d_Y(f(y),f(x))\colon d(y,x)\ge r\},
\]
and then
\[
H_f(x,r):=\frac{L_f(x,r)}{l_f(x,r)}.
\]
A homeomorphism $f\colon X\to Y$ is (metric) \emph{quasiconformal} if there is a number
$1\le H<\infty$ such that
\[
	H_f(x):=\limsup_{r\to 0} H_f(x,r)\le H
\]
for all $x\in X$.
One also defines $h_f$ by replacing ``$\limsup$'' with ``$\liminf$''.

Assuming that the spaces $X,Y$ are $Q$-dimensional, with $Q>1$, and satisfy suitable regularity assumptions,
it is known that quasiconformal mappings belong to the Newton-Sobolev space $N_{\loc}^{1,Q}(X;Y)$.
See Section \ref{sec:Prelis} for definitions.
Moreover, starting from Gehring,
in the literature there are many results saying that if $f$ satisfies a suitable relaxed version
of the metric definition of quasiconformality, then $f$ is at least in the class $N_{\loc}^{1,1}(X;Y)$.
A typical version of such a result is the following:
assume that $X,Y$ are locally compact and Ahlfors $Q$-regular,
$f\colon X\to Y$ is a homeomorphism, and $1\le p\le Q$; then if
\begin{itemize}
	\item $h_f<\infty$ outside a set $E$ with $\sigma$-finite $Q-p$-dimensional Hausdorff measure, and
	\item $h_f\in L_{\loc}^{p^*}(X)$,	
\end{itemize}
it follows that $f\in N^{1,p}_{\loc}(X;Y)$.
For this, see e.g. Williams \cite[Corollary 1.3]{Wi}.

Similarly to $H_f$, we define the pointwise asymptotic Lipschitz number
\[
\Lip_f(x):=\limsup_{r\to 0}\frac{L_f(x,r)}{r}, 
\]
and also $\lip_f$ with ``$\limsup$'' replaced by ``$\liminf$''.
If $1\le p\le \infty$
and a continuous function
$f\colon \R^n\to \R$ satisfies the following:
\begin{itemize}
	\item $\lip_f<\infty$ outside a set $E$ with $\sigma$-finite $n-1$-dimensional Hausdorff measure, and
	\item $\lip_f\in L_{\loc}^p(\R^n)$,	
\end{itemize}
then $f\in W^{1,p}_{\loc}(\R^n)$;
see Balogh--Cs\"ornyei \cite[Theorem 1.2]{BC}.
Overall, the literature studying the above two types of results is extensive, see e.g.
(in order of publication)
Gehring \cite{Ge62,Ge},
Margulis--Mostow \cite{MaMo},
Fang \cite{Fan},
Balogh--Koskela \cite{BaKo},
Kallunki--Koskela \cite{KaKo},
Heinonen--Koskela--Shanmugalingam--Tyson \cite{HKST},
Kallunki--Martio \cite{KaMa},
Koskela--Rogovin \cite{KoRo},
Balogh--Koskela--Rogovin \cite{BKR},
Wildrick--Z\"urcher \cite{WZ, Zu},
Hanson \cite{Han},
Williams \cite{Wi},
and Lahti--Zhou \cite{LaZh, LaZh2}.
Note that in these results, somewhat curiously there are two different exceptional sets.
One exceptional set is the at most $Q-p$-dimensional set $E$, whereas
the condition that $h_f$ (resp. $\lip_f$) is in $L_{\loc}^{p^*}(X)$ (resp. $L_{\loc}^{p}(X)$)
fails to give control in a set of zero $\mu$-measure, amounting to an at most $Q$-dimensional exceptional set.

In many contexts, it is natural to replace the Newton-Sobolev class $N_{\loc}^{1,1}(X;Y)$ with the 
larger class of mappings of bounded variation $\BV_{\loc}(X;Y)$.
Thus one can ask, is there a BV version of the above type of results?
We observe that just like the exceptional set $E$ in the case $p=1$, the \emph{approximate jump set}
of a BV function $f\in \BV(\R^n)$ is also of $\sigma$-finite $\mathcal H^{n-1}$-measure,
suggesting a connection between the two.
In order to obtain a natural BV result, one needs to drop the continuity assumption on $f$,
and it is also expected that one should allow a larger exceptional set $E$.
However, without continuity, already when assuming that $\Lip_f$, $H_f$ are in $L^{\infty}(\R^n)$ and finite
outside a set of $\sigma$-finite $\mathcal H^{n-1}$-measure,
it can easily happen that $f$ is not a BV mapping; see Example \ref{ex:counterexample}.

For these reasons, for a mapping $f\colon X\to Y$, we define generalized versions of $\Lip_f$ and $H_f$ as follows:
\[
\Lip_f^{\kappa,M}(x):=
\limsup_{r\to 0}\frac{L_f(x,r)}{r}\frac{\mu(B(x,Mr))}{\kappa(B(x,Mr))},\quad x\in X,
\]
and
\[
H_f^{\kappa,M}(x):=
\limsup_{r\to 0}\frac{L_f(x,r)}{l_f(x,r)}\left(\frac{\mu(B(x,Mr))}{\kappa(B(x,Mr))}\right)^{(Q-1)/Q},\quad x\in X,
\]
where $M\ge 1$ and $\kappa$ is a positive Radon measure.
Note that $\Lip_f^{\mu,1}=\Lip_f$ and $H_f^{\mu,1}=H_f$.
This type of generalized Lipschitz and distortion numbers were previously found in \cite{LahGFQ}
to be useful in various contexts.

Our main results will be given in Theorems \ref{thm:main theorem}, \ref{thm:main theorem 2}, and \ref{thm:main theorem 3}.
These theorems strive for generality, leading to rather long and complicated formulations,
but we give a simple version in the following.
The symbol $\widetilde{\mathcal H}^p$ denotes the codimension $p$ Hausdorff measure.

\begin{theorem}\label{thm:main intro}
Suppose $(X,d,\mu)$ and $(Y,d_Y,\nu)$ are Ahlfors $Q$-regular, with $Q>1$ and $1\le p\le Q$.
Let $f\colon X\to Y$ be injective and bounded, and $M\ge 1$.
Then:
\begin{enumerate}
\item If there exists a Radon measure $\kappa\ge \mu$ such that
$\Lip_f^{\kappa,M}(x)\le 1$ for $\widetilde{\mathcal{H}}^1$-a.e. $x\in X$,
then $f\in \BV_{\loc}(X;Y)$;
\item If  there exists a function $a\colon X\to [1,\infty)$ belonging to
$L_{\loc}^{p^*(Q-1)/Q}(X)$ such that
$H_f^{a,M}(x)\le 1$ for $\widetilde{\mathcal H}^p$-a.e. $x\in X$, then
$f\in N_{\loc}^{1,p}(X;Y)$.
\end{enumerate}
\end{theorem}

Besides being able to cover the BV case,
in Example \ref{ex:quasiconformal} we will see that the generalized Lipschitz and distortion numbers are
sometimes also more effective at detecting Newton-Sobolev mappings $f\in N_{\loc}^{1,p}(X;Y)$, compared with the
existing results in the literature.
Moreover, in Remark \ref{remark} and Examples \ref{ex:counterexample} and \ref{ex:set E} we will see that
at least under suitable circumstances, the exceptional set $E$
can indeed be seen to correspond to the approximate jump set of a BV function,
giving an interpretation of ``why'' one usually needs the two different exceptional sets.

\section{Notation and definitions}\label{sec:Prelis}

Throughout the paper, we consider two metric measure spaces
$(X,d,\mu)$ and $(Y,d_Y,\nu)$, where $\mu$ and $\nu$ are Borel regular outer measures.
In this paper, we only consider positive outer measures.
A ball is defined by $B(x,r):=\{y\in X\colon d(y,x)<r\}$, for $x\in X$ and $r>0$.
For a ball $B=B(x,r)$ and $M>0$, we sometimes denote $MB:=B(x,Mr)$.
In a metric space, a ball (as a set) does not necessarily have a unique center and radius, but
when using this abbreviation we will work with balls for which these have been specified.

We will always assume that $\mu$ is doubling, meaning that
there is a constant $C_d\ge 1$ such that
\[
0<\mu(B(x,2r))\le C_d\mu(B(x,r))<\infty
\]
for all $x\in X$ and $r>0$.
Given $Q\ge 1$, we say that $\mu$ (or $X$) is Ahlfors $Q$-regular 
if there is a constant $C_A\ge 1$ such that
\begin{equation}\label{eq:Ahlfors}
C_A^{-1}r^Q\le \mu(B(x,r))\le C_A r^Q
\end{equation}
for all $x\in X$ and $r>0$.

We denote the $s$-dimensional Hausdorff measure by $\mathcal H^{s}$, $s\ge 0$,
and it is obtained as a limit of the Hausdorff pre-measures 
$\mathcal H^s_{R}$ as $R\to 0$.

	We always consider $1\le p<\infty$.
	For any set $A\subset X$ and $0<R<\infty$, the restricted Hausdorff content
	of \emph{codimension} $p$ is defined by
	\[
	\widetilde{\mathcal{H}}^p_{R}(A):=\inf\left\{ \sum_{j}
	\frac{\mu(B(x_{j},r_{j}))}{r_{j}^p}\colon\,
	\ A\subset\bigcup_{j}B(x_{j},r_{j}),\ r_{j}\le R\right\},
	\]
	where we consider finite and countable coverings. 
	The codimension $p$ Hausdorff measure of $A\subset X$ is then defined by
	\[
	\widetilde{\mathcal{H}}^p(A):=\lim_{R\rightarrow 0}\widetilde{\mathcal{H}}^p_{R}(A).
	\]
	In the literature, $\widetilde{\mathcal{H}}^1$ is often denoted by $\mathcal H$.

A continuous mapping from a compact interval into $X$ is said to be a rectifiable curve if it has finite length.
A rectifiable curve $\gamma$ can always be parametrized by arc-length,
so that we get a curve $\gamma\colon [0,\ell_{\gamma}]\to X$;
see e.g. \cite[Theorem 3.2]{Haj}.
We will only consider curves that are rectifiable and
arc-length parametrized.
If $\gamma\colon [0,\ell_{\gamma}]\to X$ is a curve and
$g\colon X\to [0, \infty]$ is a Borel function, we define
\[
\int_{\gamma} g\,ds:=\int_0^{\ell_{\gamma}}g(\gamma(s))\,ds.
\]
The $p$-modulus of a family of curves $\Gamma$ is defined by
\[
\Mod_{p}(\Gamma):=\inf\int_{X}\rho^p\, d\mu,
\]
where the infimum is taken over all nonnegative Borel functions $\rho \colon X\to [0, \infty]$
such that $\int_{\gamma}\rho\,ds\ge 1$ for every curve $\gamma\in\Gamma$.
If a property holds apart from a curve family with zero $p$-modulus, we say that it holds for
$p$-a.e. curve.
Given any set $A\subset X$, we denote by $\Gamma_A$ the family of all nonconstant curves in $X$ that intersect $A$.

We will always denote by $\Om$ an open subset of $X$.

\begin{definition}\label{def:equiintegrability}
	A sequence of nonnegative
	functions $\{g_i\}_{i=1}^{\infty}$
	in $L^1(\Om)$ is equi-integrable if the following two conditions hold:
	\begin{itemize}
		\item for every $\eps>0$ there exists a $\mu$-measurable set $D\subset \Om$ such that $\mu(D)<\infty$ and
		\[
		\int_{\Om\setminus D}g_i\,d\mu<\eps \quad\textrm{for all }i\in\N;
		\]
		\item for every $\eps>0$ there exists $\delta>0$ such that if $A\subset \Om$ is $\mu$-measurable with $\mu(A)<\delta$, then
		\[
		\int_{A}g_i\,d\mu<\eps \quad\textrm{for all }i\in\N.
		\]
	\end{itemize}
\end{definition}

Next we give (special cases of) the Dunford--Pettis theorem,
Mazur's lemma, and Fuglede's lemma, see e.g.
\cite[Theorem 1.38]{AFP},
\cite[Theorem 3.12]{Rud},
and \cite[Lemma 2.1]{BB}, respectively.

\begin{theorem}\label{thm:dunford-pettis}
	Let $\{g_i\}_{i=1}^{\infty}$ be an equi-integrable sequence of nonnegative functions that is bounded in $L^1(\Om)$.
	Then there exists a subsequence such that $g_{i_j}\to g$ weakly in $L^1(\Om)$.
\end{theorem}

\begin{theorem}\label{thm:Mazur lemma}
	Let $\{g_i\}_{i=1}^{\infty}$ be a sequence
	with $g_i\to g$ weakly in $L^1(\Om)$.
	Then there exist convex combinations $\widehat{g}_i:=\sum_{j=i}^{N_i}a_{i,j}g_j$,
	for some $N_i\in\N$,
	such that $\widehat{g}_i\to g$ in $L^1(\Om)$.
\end{theorem}

By convex combinations we mean that the numbers $a_{i,j}$ are nonnegative
and that $\sum_{j=i}^{N_i}a_{i,j}=1$ for every $i\in\N$.

\begin{lemma}\label{lem:Fuglede lemma}
	Let $\{g_i\}_{i=1}^{\infty}$ be a sequence of functions with $g_i\to g$ in $L^p(\Om)$.
	Then for $p$-a.e. curve $\gamma$ in $\Om$, we have
	\[
	\int_{\gamma}g_i\,ds\to \int_{\gamma}g\,ds\quad\textrm{as }i\to\infty.
	\]
\end{lemma}

By a Radon measure $\kappa$ we mean a locally finite Borel regular outer measure;
locally finite means that for every $x\in X$ there is $r>0$ such that $\kappa(B(x,r))<\infty$.
We will also need the following Vitali-Carath\'eodory theorem;
for a proof see e.g. \cite[p. 108]{HKSTbook}.
Recall that $1\le p<\infty$.

\begin{theorem}\label{thm:VitaliCar}
	Let $\kappa$ be a Radon measure on $\Om$ and let $\rho\in L^p(\Om,\kappa)$ be nonnegative.
	Then there exists a sequence $\{\rho_i\}_{i\in \N}$
	of lower semicontinuous functions
	on $\Om$ such that $\rho\le \rho_{i+1}\le\rho_i$ for all $i\in\N$, and
	$\rho_i\to\rho$ in $L^p(\Om,\kappa)$.
\end{theorem}

\begin{definition}\label{def:upper gradient}
	Let $f\colon \Om\to Y$. We say that a Borel
	function $g\colon \Om\to [0,\infty]$ is an upper gradient of $f$ in $\Om$ if
	\begin{equation}\label{eq:upper gradient inequality}
		d_Y(f(\gamma(0)),f(\gamma(\ell_{\gamma})))\le \int_{\gamma}g\,ds
	\end{equation}
	for every curve $\gamma$ in $\Om$.
	If $g\colon \Om\to [0,\infty]$ is a $\mu$-measurable function
	and (\ref{eq:upper gradient inequality}) holds for $p$-a.e. curve in $\Om$,
	we say that $g$ is a $p$-weak upper gradient of $f$ in $\Om$.
\end{definition}

We write $f\in L^1(\Om;Y)$ if $d_Y(f(\cdot),f(x))\in L^1(\Om)$ for some $x\in \Om$.

\begin{definition}
	The Newton-Sobolev class $N^{1,p}(\Om;Y)$ consists of those
	mappings $f\in L^p(\Om;Y)$ for which there exists a $p$-weak upper gradient $g\in L^p(\Om)$ of $f$ in $\Om$.
	
	The Dirichlet class $D^p(\Om;Y)$ consists of those mappings $f\colon \Om\to Y$
	for which there exists a $p$-weak upper gradient $g\in L^p(\Om)$ of $f$ in $\Om$.
	We define
	\[
	\Vert f\Vert_{D^p(\Om;Y)}:= \inf \Vert g\Vert_{L^p(\Om)},
	\]
	where the infimum is taken over all $p$-weak upper gradients $g$ of $f$ in $\Om$.
\end{definition}

One can define these classes also by using upper gradients instead of $p$-weak upper gradients,
but this leads to the same result, see \cite[Lemma 1.46]{BB}.
In the classical setting of $X=Y=\R^n$, both spaces equipped with the Lebesgue measure,
we have that functions in the class $N^{1,p}(X;Y)$ are exactly suitable pointwise representatives of functions
in the classical Sobolev class $W^{1,p}(\R^n;\R^n)$, see e.g. \cite[Theorem A.2]{BB}.
The Newton-Sobolev class in metric spaces was first introduced by Shanmugalingam \cite{Shan}.

To define the class of BV mappings,
we consider the following definitions by Martio,
who studied BV functions on metric spaces \cite{Mar}.
Given a family of curves $\Gamma$, we say that a sequence of nonnegative Borel functions $\{\rho_i\}_{i=1}^{\infty}$
is $\AM$-admissible for $\Gamma$ if
\[
\liminf_{i\to\infty}\int_{\gamma}\rho_i\,ds\ge 1\quad \textrm{for all }\gamma\in\Gamma.
\]
Then we let
\[
\AM(\Gamma):=\inf \left\{\liminf_{i\to\infty}\int_X \rho_i\,d\mu\right\},
\]
where the infimum is taken over all $\AM$-admissible sequences $\{\rho_i\}_{i=1}^{\infty}$.
Note that always $\AM(\Gamma)\le \Mod_1(\Gamma)$.
Then, BV mappings are defined as follows.

\begin{definition}\label{def:BV def}
	Given a mapping $f\colon \Om\to Y$, we say that $f$ is in the Dirichlet class
	$D^{\BV}(\Om;Y)$ if there exists a sequence of
	nonnegative functions $\{g_i\}_{i=1}^{\infty}$ that is bounded in $L^1(\Om)$, such that
	for $\AM$-a.e. curve $\gamma$ in $\Om$, we have
	\begin{equation}\label{eq:AM ae curve}
		d_Y(f(\gamma(t_1)),f(\gamma(t_2)))
		\le \liminf_{i\to\infty}\int_{\gamma|_{[t_1,t_2]}}g_i\,ds
	\end{equation}
	for almost every $t_1,t_2\in [0,\ell_{\gamma}]$ with $t_1<t_2$.
	We also define
	\[
	\Vert Df\Vert(\Om):=\inf \left\{\liminf_{i\to\infty}\int_{\Om}g_i\,d\mu\right\},
	\]
	where the infimum is taken over sequences $\{g_i\}_{i=1}^{\infty}$ as above.
	If also $f\in L^1(\Om;Y)$, then we say that $f\in \BV(\Om;Y)$.
\end{definition}

Durand-Cartagena--Eriksson-Bique--Korte--Shanmugalingam \cite{DEKS} show that in the case where $f$ is real-valued,
$X$ is complete, and $\mu$ is a doubling measure that supports a
Poincar\'e inequality, Definition \ref{def:BV def} agrees with Miranda's definition of BV functions
in metric spaces given in \cite{Mir},
which in turn agrees in Euclidean spaces with the classical definition.
Thus Definition \ref{def:BV def} is natural for us to use.

Given a $\mu$-measurable set $A\subset X$ with $0<\mu(A)<\infty$ and a function $u\in L^1(A)$, we denote the integral average by
\[
\vint{A}u\,d\mu:=\frac{1}{\mu(A)}\int_A u\,d\mu.
\]

In the Euclidean space $\R^n$, with $n\ge 1$, BV functions $f\in\BV(\R^n;\R^n)$
and their variation measures have the following structure; see \cite[Section 3]{AFP}.
Given $x\in\R^n$, $r>0$, and a unit vector $\nu\in \R^n$, we define the half-balls
\begin{align*}
	B_{\nu}^+(x,r)\coloneqq \{y\in B(x,r)\colon \langle y-x,\nu\rangle>0\},\\
	B_{\nu}^-(x,r)\coloneqq \{y\in B(x,r)\colon \langle y-x,\nu\rangle<0\},
\end{align*}
where $\langle \cdot,\cdot\rangle$ denotes the inner product.
We consider the Euclidean space equipped with the Euclidean metric and $\mu=\mathcal L^n$,
that is, the $n$-dimensional Lebesgue measure.
We say that $x\in \R^n$ is an approximate jump point of $f$ if there exist a unit vector $\nu\in \R^n$
and distinct vectors $f^+(x), f^-(x)\in\R^n$ such that
\begin{equation}\label{eq:jump value 1}
	\lim_{r\to 0}\,\vint{B_{\nu}^+(x,r)}|f(y)-f^+(x)|\,d\mathcal L^n(y)=0
\end{equation}
and
\[
	\lim_{r\to 0}\,\vint{B_{\nu}^-(x,r)}|f(y)-f^-(x)|\,d\mathcal L^n(y)=0,
\]
where $|\cdot|$ is the Euclidean norm in $\R^n$.
The set of all approximate jump points is denoted by $J_f$.
We write the Radon-Nikodym decomposition of the variation measure of $f$ into the absolutely continuous and singular parts with respect to $\mathcal L^n$
as $\Vert Df\Vert=\Vert Df\Vert^a+\Vert Df\Vert^s$. Furthermore, we define the Cantor and jump parts of $Df$ by
\[
	\Vert Df\Vert^c\coloneqq  \Vert Df\Vert^s\mres (\R^n\setminus J_f),\quad \Vert Df\Vert^j
	\coloneqq \Vert Df\Vert^s\mres J_f.
\]
Here
\[
\Vert Df\Vert^s \mres J_f(A):=\Vert Df\Vert^s (J_f\cap A),\quad \textrm{for } \Vert Df\Vert^s\textrm{-measurable } A\subset \R^n.
\]
We get the decomposition
\[
	\Vert Df\Vert=\Vert Df\Vert^a + \Vert Df\Vert^c+ \Vert Df\Vert^j.
\]
For the jump part, we know that
(see \cite[Section 3.9]{AFP})
\begin{equation}\label{eq:jump part representation}
	d\Vert Df\Vert^j=|f^{+}-f^-|\,d\mathcal H^{n-1}\mres J_f.
\end{equation}
Thus the jump set $J_f$ is $\sigma$-finite with respect to $\mathcal H^{n-1}$, that is, it can be written as an
at most countable union of sets of finite $\mathcal H^{n-1}$-measure.

A rather similar decomposition holds also in more general metric measure spaces, but we will not use this.\\

Consider $f\colon \Om\to Y$.
For every $x\in \Om$ and $r>0$, we define
\[
L_f(x,r):=\sup\{d_Y(f(y),f(x))\colon  d(y,x)\le r,\, y\in \Om\}
\]
and
\[
l_f(x,r):=\inf\{d_Y(f(y),f(x))\colon d(y,x)\ge r,\, y\in \Om\}.
\]

Next, we define the generalized Lipschitz and distortion numbers.

\begin{definition}
	Consider a mapping $f\colon \Om\to Y$, a Radon measure $\kappa$ on $\Om$,
	and $M\ge 1$. For every $x\in\Om$, we define
	\[
	\Lip_f^{\kappa,M}(x):=
	\limsup_{r\to 0}\frac{L_f(x,r)}{r}\frac{\mu(B(x,Mr))}{\kappa(B(x,Mr))}.
	\]
	We also define
	\[
	H_f^{\kappa,M}(x):=
	\limsup_{r\to 0}\frac{L_f(x,r)}{l_f(x,r)}\left(\frac{\mu(B(x,Mr))}{\kappa(B(x,Mr))}\right)^{(Q-1)/Q}.
	\]
\end{definition}

Whenever a denominator is zero, we consider the corresponding quantity to be $\infty$.
However, we will only consider Radon measures $\kappa$ for which $\kappa(B(x,r))>0$ for every 
ball $B(x,r)\subset \Om$.
In the case $d\kappa:=a\,d\mu$ for a nonnegative function $a\in L_{\loc}^1(\Om)$,
we also denote $\Lip_f^{a,M}:=\Lip_f^{\kappa,M}$ and $H_f^{a,M}:=H_f^{\kappa,M}$.
We also define, as in the Introduction, $\Lip_f:=\Lip_f^{\mu,1}$ and $H_f:=H_f^{\mu,1}$,
and $\lip_f$ and $h_f$ analogously but with ``$\limsup$'' replaced by ``$\liminf$''.

\begin{remark}\label{remark}
	We will show that the generalized Lipschitz and distortion numbers can be used to give sufficient conditions
	for $f$ to be a Newton-Sobolev or BV mapping. The definitions are also motivated by the fact that
	in the Euclidean setting, a variant of
	$\Lip_f^{a,M}$ can in fact be used to \emph{characterize} Sobolev functions
	$f\in W^{1,p}(\R^n;\R^n)$
	(see \cite[Theorem 1.3]{LahGFQ}), and a different generalized distortion number can be used to detect
	the rank of $\tfrac {dDf}{d|Df|}$ when $f\in \BV(\R^n;\R^n)$
	(see \cite[Theorem 6.3]{LahGFQ}).
	In particular, it follows from the proof of \cite[Theorem 6.3]{LahGFQ} that
	by choosing a natural pointwise representative of $f$, we have
	$H_f=\infty$ $\mathcal H^{n-1}$-a.e.
	in the approximate jump set $J_f$.
\end{remark}

We will not always assume Ahlfors regularity \eqref{eq:Ahlfors}, but nonetheless
we understand $Q>1$ to be a fixed parameter throughout the paper. Given any $1\le p\le Q$,
we denote the Sobolev conjugate by $p^*=Qp/(Q-p)$ when $p<Q$,
and $p^*=\infty$ when $p=Q$.
Our standing assumptions will be the following.\\

\emph{Throughout this paper we assume that $(X,d,\mu)$
	and $(Y,d_Y,\nu)$ are metric spaces equipped with Borel regular outer measures $\mu$ and $\nu$,
	such that $\mu$ is doubling,
	 every ball in $Y$ has nonzero and finite measure, $X$ is connected, and $1<Q<\infty$ is fixed.
}

\section{Preliminaries}

In this section we consider some preliminary examples and results.
Recall that $\Om$ always denotes an open subset of $X$. 

We start with the following simple example.

\begin{example}\label{ex:counterexample}
Let $(X,d,\mu)$ be $\R^2$ equipped with the Euclidean metric and Lebesgue measure.
Let $\Om:=(-1,1)\times (0,1)$ and
\[
f(x_1,x_2):=
\begin{cases}
	(x_1,x_2) \quad\textrm{when }-1<x_1\le 0\ \textrm{ or }\ 1/(j+1) \le x_1< 1/j,\ j\in\N\textrm{ is odd,}\\
		(2-x_1,x_2) \quad\textrm{when }1/(j+1) \le x_1< 1/j,\ j\in\N\textrm{ is even.}\\
\end{cases}
\]
Now $f\colon \Om\to \R^2$ is injective and bounded, and the set
\begin{align*}
E
&:=\{\Lip_f=\infty\}=\{H_f=\infty\}\\
&=\{(x_1,x_2)\in\R^2\colon x_1=0,\,0<x_2<1\}\cup
\bigcup_{j=2}^{\infty}\{(x_1,x_2)\in\R^2\colon x_1=1/j,\,0<x_2<1\}
\end{align*}
has $\sigma$-finite $1$-dimensional
Hausdorff measure.
Note that the codimension 1 Hausdorff measure $\widetilde{\mathcal{H}}^1$ is now comparable to the
$1$-dimensional Hausdorff measure $\mathcal H^1$.
Moreover, in the set $\Om\setminus E$ we have $\Lip_f=H_f=1$ ,
so that $\Lip_f,H_f\in L^p(\Om)$ for all $1\le p\le \infty$.
But $f$ is not even in $\BV_{\loc}(\Om,\R^2)$; this can be seen e.g. from the representation \eqref{eq:jump part representation}.
Note also that $E=J_f$ (the approximate jump set of $f$).
\end{example}

This example indicates that it is difficult to formulate a reasonable
sufficient condition for $f$ belonging to
the BV class by only using $\Lip_f$ or $H_f$.
Hence we are motivated to define the generalized versions of these quantities.

The following lemma will be used for upper gradient estimates on curves.
Note that despite the simplicity of the lemma, one has to be careful with the fact that $h$ is not
assumed to be continuous, and in fact the lemma would fail if we did not assume the sets $U_{j,k}$
to be open.

\begin{lemma}\label{lem:line behavior}
	Let  $h\colon [a,b]\to Y$ for some finite interval $[a,b]\subset \R$.
	Suppose that there is a sequence of at most countable unions of sets $W_j=\bigcup_{k}U_{j,k}$, $j\in\N$,
	where each $U_{j,k}\subset \R$ is open and bounded, 
	and such that $\ch_{W_j}(t)\to 1$ as $j\to\infty$ for all $t\in [a,b]$.
	Then
	\[
	d_Y(h(a),h(b))\le \liminf_{j\to\infty}\sum_k\diam h(U_{j,k}\cap [a,b]).
	\]
\end{lemma}

Here $\ch_{W_j}$ is the characteristic function of $W_j$, that is, $\ch_{W_j}=1$ in $W_j$ and $\ch_{W_j}=0$ outside
$W_j$. Moreover,
$\diam h(U_{j,k}\cap [a,b]):=\sup_{s,t\in U_{j,k}\cap [a,b]}d_Y(h(s),h(t))$.

\begin{proof}
	In \cite[Lemma 3.9]{LahGFQ} this was shown for real-valued $h$.
	By the Kuratowski embedding theorem,
	every metric space, in particular $Y$, can be isometrically embedded into a Banach space,
	for example $V:=L^\infty(Y)$; see e.g. \cite[p. 100]{HKSTbook}.
	Denote the norm in $V$ by $\Vert \cdot\Vert_V$, and the dual of $V$ by $V^*$.
	We find $v^*\in V^*$ with $\Vert v^*\Vert_{V^*}=1$ and $v^*(h(a)-h(b))=\Vert h(a)-h(b)\Vert_V$.
	Now we can consider the real-valued function $v^*(h(\cdot))$, and we get
	\begin{align*}
	d_Y(h(a),h(b))
	&=\Vert h(a)-h(b)\Vert_V\\
	&= v^*(h(a)-h(b))\\
	&= v^*(h(a))-v^*(h(b))\\
	&\le \liminf_{j\to\infty}\sum_k\diam v^*(h(U_{j,k}\cap [a,b]))\quad\textrm{by \cite[Lemma 3.9]{LahGFQ}}\\
	&\le \liminf_{j\to\infty}\sum_k\diam h(U_{j,k}\cap [a,b]).
	\end{align*}
\end{proof}

We will also need the following measurability result.

\begin{lemma}\label{lem:Borel for lip f}
Suppose $f\colon \Om\to Y$ is continuous at every point $x\in H\subset \Om$,
with $\mu(\Om\setminus H)=0$. Then $\lip_f$ is a $\mu$-measurable function on $\Om$.
\end{lemma}
\begin{proof}
For a fixed $r>0$, the mapping
\[
x \mapsto \frac{\sup\{d_Y(f(y),f(x))\colon d(y,x)< r,\,y\in\Om\}}{r}
\]
is lower semicontinuous at every point $x\in H$.
On the other hand, for every $x\in \Om$,
\begin{align*}
\lip_f(x)
=\liminf_{r\to 0}\frac{L_f(x,r)}{r}
&=\liminf_{r\to 0}\frac{\sup\{d_Y(f(y),f(x))\colon d(y,x)\le  r,\,y\in\Om\}}{r}\\
&=\liminf_{r\to 0}\frac{\sup\{d_Y(f(y),f(x))\colon d(y,x)< r,\,y\in\Om\}}{r}\\
&=\lim_{j\to\infty}\inf_{r\in (0,1/j)\cap \mathcal Q}\frac{\sup\{d_Y(f(y),f(x))\colon d(y,x)< r,\,y\in\Om\}}{r}.
\end{align*}
Together these establish that $\lip_f$ is a Borel function in $H$.
Thus it is $\mu$-measurable in $\Om$.
\end{proof}

The next lemma is standard; see e.g. \cite[Lemma 4.8]{LaZh2}.

\begin{lemma}\label{lem:abs cont initial}
	Let $-\infty<a<b<\infty$.
	Let $h\colon [a,b]\to Y$ be a continuous mapping such that $\lip_h(t)<\infty$
	for all $t\in T$ for some $\mathcal L^1$-measurable set $T\subset [a,b]$.
	Then
	\[
	\mathcal H^1(h(T))\le \int_T\lip_h(t)\,dt.
	\]
\end{lemma}

The following lemma says essentially that $\lip_f$ is a weak upper gradient of $f$.
Recall that $1\le p<\infty$.

\begin{lemma}\label{lem:abs cont}
	Suppose $f\colon \Om\to Y$ is continuous at $\mu$-a.e. point in $\Om$.
	Then for $p$-a.e. curve $\gamma\colon [0,\ell_{\gamma}]\to \Om$ we have that
	if $f\circ\gamma$ is continuous and
	$\lip_f(\gamma(t))<\infty$ for all $t\in T\subset [0,\ell_{\gamma}]$ where $T$ is $\mathcal L^1$-measurable,
	then we get
	\[
	\mathcal H^1(f(\gamma(T)))\le \int_{T}\lip_f(\gamma(t))\,dt.
	\]
\end{lemma}
\begin{proof}
	First note that $\lip_f$ is $\mu$-measurable in $\Om$ by Lemma \ref{lem:Borel for lip f}.
	Then curve integrals of $\lip_f$ are well-defined on $p$-a.e. curve in $\Om$,
	see e.g. Lemma 1.43 and its proof in \cite{BB}.
	Consider such a curve $\gamma$.
	For the mapping $h:=f\circ \gamma\colon [0,\ell_{\gamma}]\to Y$, by the fact that $\gamma$
	is a $1$-Lipschitz mapping, for all $t\in T$ we have
	\begin{align*}
		\lip_h(t)
		&=\liminf_{r\to 0}\sup_{s\in B(t,r)}\frac{d_Y(f\circ \gamma(s),f\circ \gamma(t))}{r}\\
		&\le \liminf_{r\to 0}\sup_{y\in B(\gamma(t),r)}\frac{d_Y\big(f(y),f (\gamma(t))\big)}{r}\\
		&= \lip_f(\gamma(t)).
	\end{align*}
	Now the result follows from Lemma \ref{lem:abs cont initial}.
\end{proof}

Since $f$ is generally not assumed to be continuous in this paper, it is of interest to
note that $f$ nonetheless always has the following weak type of continuity.

\begin{lemma}\label{lem:liminf continuity}
	Let $f\colon \Om\to Y$. Then for all except at most countably many $x\in \Om$, there exists
	a sequence $y_l\to x$, $y_l\neq x$,
	with $f(y_l)\to f(x)$.
\end{lemma}
\begin{proof}
	Consider the set of points in $\Om$ for which no such sequence exists; call this set $A$.
	Let $x\in A$. Then there exists $\delta>0$ such that
	$d_Y(f(y),f(x))\ge \delta$ for all $y\in \Om\cap B(x,\delta)\setminus \{x\}$.
	Thus $A=\bigcup_{j=1}^{\infty}A_j$ with
	\[
	A_j:=\{x\in \Om\colon d_Y(f(y),f(x))\ge 1/j\ \textrm{ for all }\ y\in \Om\cap B(x,1/j)\setminus \{x\}\}.
	\]
	It is enough to show that every $A_j$ is countable; suppose instead that $A_j$ is
	uncountable for a fixed $j$.
	Note that both $X$ and $Y$ are separable;
	see \cite[Proposition 1.6]{BB}.
	Since $X$ is separable, there exists a collection of points $\{z_k\}_{k=1}^{\infty}\subset \Om$
	that is dense in $\Om$. Thus $\Om$ is covered by the countable collection of balls
	$\{B(z_k,(2j)^{-1})\}_{k=1}^{\infty}$. Then for some $k\in\N$, also the set
	$D:=B(z_k,(2j)^{-1})\cap A_j$
	is uncountable. If $x_1,x_2\in D$ are two distinct points, then $d(x_1,x_2)<1/j$,
	and so $d_Y(f(x_1),f(x_2))\ge 1/j$.
	Then the balls $B(f(x),(2j)^{-1})$, $x\in D$, are disjoint.
	But this contradicts the separability of $Y$.
\end{proof}

Denote by $D_f$ the set of points in $\Om$ where $f$ is not continuous.

\begin{lemma}\label{lem:discontinuity set}
	Let $f\colon \Om\to Y$. Then $D_f\cap\{H_f<\infty\}$ is an at most countable set.
\end{lemma}
\begin{proof}
	Let $x\in \Om$ be a point that is not in the exceptional set of Lemma \ref{lem:liminf continuity}.
	Then we find a sequence $y_k\to x$, $y_k\neq x$, with
	$d_Y(f(y_k),f(x))\to 0$.
	Suppose $f$ is discontinuous at $x$; it is enough to show that $H_f(x)=\infty$.
	Now we also find a sequence 
	$z_k\to x$ with $d_Y(f(z_k),f(x))\ge \delta>0$.
	Fix an arbitrary $M>0$.
	For all sufficiently small $r>0$, there exists $y_k\in \Om$ with $d(y_k,x)>r$ and $d_Y(f(y_k),f(x))< \delta/M$,
	and $z_k\in \Om$ with $d(z_k,x)<r$ and $d_Y(f(z_k),f(x))\ge  \delta$.
	Thus
	\[
	\frac{L_f(x,r)}{l_f(x,r)}\ge \frac{ \delta}{ \delta/M}=M.
	\]
	It follows that
	\[
	\limsup_{r\to 0}H_f(x,r)\ge M\quad\textrm{and thus}\quad \limsup_{r\to 0}H_f(x,r)=\infty,
	\]
	since $M>0$ was arbitrary. Thus $H_f(x)=\infty$.
\end{proof}

Finally, we give the following lemma.
Recall that we denote by $\Gamma_A$ the
family of all nonconstant curves in $X$ that intersect $A$, and that $1\le p<\infty$.

\begin{lemma}\label{lem:curve intersects set}
Let $A\subset X$ with $\widetilde{\mathcal H}^p(A)=0$.
Then $\Mod_p(\Gamma_A)=0$.
\end{lemma}
\begin{proof}
	Fix $\eps>0$.
For every $j\in\N$, we find an
(at most countable) collection of balls $B_{j,k}=B(x_{j,k},r_{j,k})$ covering $A$,
such that $r_{j,k}<1/j$ and
\[
\sum_{k}\frac{\mu(B_{j,k})}{r_{j,k}^p}<2^{-j}\eps.
\]
Define
\[
\rho:=\sup_{j,k}\frac{\ch_{2B_{j,k}}}{r_{j,k}}.
\]
For every $\gamma \in \Gamma_A$, we can choose $j>1/\ell_{\gamma}$ and
$k$ such that $\gamma\cap B_{j,k}\neq \emptyset$, and then
\[
\int_{\gamma}\rho\,ds\ge \int_{\gamma}\frac{\ch_{2B_{j,k}}}{r_{j,k}}\,ds\ge 1.
\]
Now
\begin{align*}
\Mod_p(\Gamma_A)\le\Vert \rho\Vert_{L^p(X)}^p
= \int_X \left(\sup_{j,k}\frac{\ch_{2B_{j,k}}}{r_{j,k}}\right)^p\,d\mu
&\le \sum_{j,k}\frac{\mu(2B_{j,k})}{r_{j,k}^p}  \le C_d \eps.
\end{align*}
Letting $\eps\to 0$, we get the result.
\end{proof}

\section{Main theorems}

In this section we state and prove our main theorems.
The theorems all take a rather long and complicated form, but in Section \ref{sec:cor} we will consider more
streamlined corollaries.
In \cite[Theorem 5.10]{LahGFQ}, a somewhat similar theorem was proved in the Euclidean setting, but there
(as in all of the literature, as far as we know)
the mapping $f$ was assumed to be continuous, which is of course not a natural
assumption in the BV setting.

Recall that $1<Q<\infty$ is a fixed parameter,
and that we denote by $\Gamma_N$ the
family of all nonconstant curves in $X$ that intersect $N$.

\begin{theorem}\label{thm:main theorem}
	Suppose $\Om$ is nonempty, open, and bounded,
	$f\colon\Om\to Y$,
	$M\ge 1$, $\kappa$ is a finite Radon measure on $\Om$ whose support contains $\Om$,
	and suppose $\Om$ admits a partition into disjoint sets $A$, $D$, and $N$ such that
	$\Lip_f^{\kappa,M}<\infty$ in $A$,
	$H_f^{\kappa,M}<\infty$ in $D$,
	$\Mod_1(\Gamma_N)=0$, and
	\begin{equation}\label{eq:A1A2 assumption}
		\int_{\Om}h\,d\kappa<\infty
		\quad\textrm{for a function }h\ge \Lip_f^{\kappa,M}\ch_{A}+(H_f^{\kappa,M})^{Q/(Q-1)}\ch_D.
	\end{equation}
	Furthermore, suppose that for some $0<C_{\Om}<\infty$ and for all $r>0$,
	\begin{equation}\label{eq:measure lower and upper bound}
	\frac{\mu(B(x,r))}{r^Q}\le C_{\Om}
	\ \textrm{ for all }x\in \Om
	\quad\textrm{and}\quad
	\frac{\nu(B(f(x),r))}{r^Q}\ge C_{\Om}^{-1}
	\ \textrm{ for all }x\in D,
	\end{equation}
	that there is an open set $W$ such that $D\subset W\subset \Om$
	and $f$ is injective in $W$, and that there is $\beta>0$ such that
	\begin{equation}\label{eq:finite image assumption}
		\nu(V)<\infty\quad\textrm{for }V:=\bigcup_{y\in D}B(f(y),l_f(y,\beta)).
	\end{equation}
	Also let $0<\eps\le 1$. Then $f\in D^{\BV}(\Om;Y)$ with
	\begin{equation}\label{eq:BV regularity claim}
	\Vert Df\Vert(\Om) \le
	C_1\eps^{-1/(Q-1)}\int_{\Om}h\,d\kappa
	+C_1\eps(\kappa(\Om)+\nu(V)),
	\end{equation}
	where $C_1$ only depends on $Q$, $M$, $C_d$, and $C_{\Om}$.
\end{theorem}

Here $\lceil b\rceil$ denotes the smallest integer at least $b\in\R$.

\begin{proof}
	By the Vitali-Carath\'eodory theorem (Theorem \ref{thm:VitaliCar}), we can assume that the function $h$
	is lower semicontinuous.
	Denote
	\begin{equation}\label{eq:Omega delta def}
	\Om(\delta):=\{x\in \Om\colon \dist(x,X\setminus \Om)>\delta\},\quad\delta>0.
	\end{equation}
	Here $\dist(x,X\setminus \Om):=\inf \{d(x,y)\colon y\in X\setminus \Om\}$.
	We have $A=\bigcup_{j=1}^{\infty}A_j$, where $A_j$ consists of those points
	$x\in A\cap \Om(j^{-1}(M+1))$ for which
	\begin{equation}\label{eq:Aj def}
	\sup_{0<r\le 1/j}\frac{L_f(x,r)}{r}\frac{\mu(B(x,Mr))}{\kappa(B(x,Mr))}
	\le \Lip_f^{\kappa,M}(x)+\eps^{Q/(Q-1)}
	\end{equation}
	and
	\begin{equation}\label{eq:Aj def 2}
	\Lip_f^{\kappa,M}(x)\le h(z)+\eps^{Q/(Q-1)}
	\quad\textrm{for all }z\in B(x,M/j).
	\end{equation}
	Using the fact that $\mu$ is doubling, for any fixed $j\in\N$ we find an (at most countable) collection $\{B_{j,k}=B(x_{j,k},1/j)\}_{k}$ of balls covering $A_j$, with $x_{j,k}\in A_j$,
	and such that the balls $(M+1)B_{j,k}\subset\Om$
	can be divided into at most
	\[
	C_{M}:= C_d^{\lceil \log_2(18M)\rceil }
	\]
	collections of pairwise disjoint balls.
	
	Analogously, we have $D=\bigcup_{j=1}^{\infty}D_j$, where $D_j$ consists of those points
	$y\in D\cap W(j^{-1}(M+1))$ for which
	\begin{equation}\label{eq:Dj def}
	\sup_{0<s\le 1/j}\frac{L_f(y,s)}{l_f(y,s)}\left(\frac{\mu(B(y,Ms))}{\kappa(B(y,Ms))}\right)^{(Q-1)/Q}
	\le H_f^{\kappa,M}(y)+\eps
	\end{equation}
	and
	\begin{equation}\label{eq:Dj def 2}
	(H_{f}^{\kappa,M}(y))^{Q/(Q-1)}\le h(z)+\eps^{Q/(Q-1)}
	\quad\textrm{for all }z\in B(y,M/j).
	\end{equation}
	We find a collection $\{\widehat{B}_{j,k}=B(y_{j,k},1/j)\}_{k}$ of balls
	covering $D_j$, with
	$y_{j,k}\in D_j$,
	and such that the balls $(M+1)\widehat{B}_{j,k}\subset W$ can be divided into at most $C_{M}$
	collections of pairwise disjoint balls.
	
	Denote $L_{j,k}:=L_f(x_{j,k},1/j)$,
	$\widehat{L}_{j,k}:=L_f(y_{j,k},1/j)$, and $\widehat{l}_{j,k}:=l_f(y_{j,k},1/j)$.
	For each $j\in\N$, define
	\[
	g_j:=2j\sum_{k}L_{j,k}\ch_{2B_{j,k}} +2j\sum_{k} \widehat{L}_{j,k} \ch_{2\widehat{B}_{j,k}}.
	\]
	By assumption,
	$1$-a.e. nonconstant curve in $\Om$ has empty intersection with $N$.
	Take such a curve $\gamma\colon [0,\ell_{\gamma}]\to\Om$, with $\ell_{\gamma}>0$.
	By a slight abuse of notation, denote also the image of a curve $\gamma$ by the same symbol.
	Now we have
	\[
	\gamma\subset A\cup D=\bigcup_{j=1}^{\infty}A_j \cup \bigcup_{j=1}^{\infty}D_j,
	\]
	where from the definitions we can see that $A_j\subset A_{j+1}$ and $D_j\subset D_{j+1}$ for every $j\in\N$.
	Thus
	\[
	\lim_{j\to\infty}\ch_{\bigcup_{k}B_{j,k}}(x)=1\ \ \textrm{for every }x\in A
	\quad\textrm{and}\quad
	\lim_{j\to\infty}\ch_{\bigcup_{k}\widehat{B}_{j,k}}(y)=1\ \ \textrm{for every }y\in D.
	\]
	If $\ell_{\gamma}\ge 1/j$ and $\gamma\cap B_{j,k}\neq \emptyset$, then
	\[
	2j\int_\gamma L_{j,k}\ch_{2B_{j,k}}\,ds
	\ge 2L_{j,k}
	\ge \diam f(B_{j,k}).
	\]
	The analogous fact holds for the balls $\widehat{B}_{j,k}$.
	Thus we get
	\begin{align*}
	\int_\gamma g_j\,ds	
	&\ge 2j\sum_{k\colon \gamma\cap B_{j,k}\neq \emptyset}\int_\gamma L_{j,k} \ch_{2B_{j,k}}\,ds
	+2j\sum_{k\colon \gamma\cap \widehat{B}_{j,k}\neq \emptyset}\int_\gamma \widehat{L}_{j,k} \ch_{2\widehat{B}_{j,k}}\,ds\\
	&\ge \sum_{k\colon \gamma\cap B_{j,k}\neq \emptyset}\diam f(B_{j,k})
	+\sum_{k\colon \gamma\cap \widehat{B}_{j,k}\neq \emptyset}\diam f(\widehat{B}_{j,k}).
	\end{align*}
	Applying Lemma \ref{lem:line behavior} with $h:=f\circ\gamma\colon [0,\ell_{\gamma}]\to Y$
	(assuming only $\ell_{\gamma}>0$), we get
	\begin{equation}\label{eq:ug estimate}
		\begin{split}
		&d_Y\big(f(\gamma(0)),f(\gamma(\ell_{\gamma}))\big)\\
		&\quad \le \liminf_{j\to\infty}\left(\sum_{k\colon \gamma\cap B_{j,k}\neq \emptyset}\diam f\circ\gamma(\gamma^{-1}(B_{j,k}))
		+\sum_{k\colon \gamma\cap \widehat{B}_{j,k}\neq \emptyset}\diam f\circ\gamma(\gamma^{-1}(\widehat{B}_{j,k}))\right)\\
		&\quad \le \liminf_{j\to\infty}\left(\sum_{k\colon \gamma\cap B_{j,k}\neq \emptyset}\diam f(B_{j,k})
		+\sum_{k\colon \gamma\cap \widehat{B}_{j,k}\neq \emptyset}\diam f(\widehat{B}_{j,k})\right)\\
		&\quad \le \liminf_{j\to\infty}\int_\gamma g_j\,ds.
		\end{split}
	\end{equation}
On the other hand, we estimate
\begin{equation}\label{eq:A1 estimate}
	\begin{split}
		2j\sum_{k} L_{j,k}\mu(2B_{j,k})
		&\le 2 C_d \sum_{k}\kappa(MB_{j,k})\big(\Lip_f^{\kappa,M}(x_{j,k})+\eps^{Q/(Q-1)}\big)\quad\textrm{by }\eqref{eq:Aj def}\\
		&\le 2 C_d \sum_{k}\int_{MB_{j,k}}(h+2\eps^{Q/(Q-1)})\,d\kappa\quad\textrm{by }\eqref{eq:Aj def 2}\\
		&\le 2 C_d C_M\int_{\Om}(h+2\eps^{Q/(Q-1)})\,d\kappa.
	\end{split}
\end{equation}
In general, note that if balls $B_1=B(z_1,r_1)$ and $B_2=B(z_2,r_2)$ are disjoint
and contained in $W$, then
from the definition of $l_f(\cdot,\cdot)$ and from the injectivity of $f$ it follows that
\begin{equation}\label{eq:injectivity for balls}
B(f(z_1),l_f(z_1,r_1))\cap B(f(z_2),l_f(z_2,r_2)) = \emptyset.
\end{equation}
Note also that the collection $\{B(y_{j,k},1/j)\}_{k=1}^{\infty}$ can be divided into at most
$C_M$ collections of pairwise disjoint balls.
In total, when $1/j< \beta$ we get
\begin{equation}\label{eq:finite nu measure estimate}
\sum_{k}\nu(B(f(y_{j,k}),\widehat{l}_{j,k}))\le C_M\nu(V)<\infty
\end{equation}
by \eqref{eq:finite image assumption}.
With $1/Q+(Q-1)/Q=1$,
by Young's inequality we have for any $b_1,b_2\ge 0$ that
\begin{equation}\label{eq:generalized Young}
	b_1b_2=\eps^{1/Q} b_1 \eps^{-1/Q} b_2
	\le \frac{1}{Q}\eps b_1^Q + \frac{Q-1}{Q}\eps^{-1/(Q-1)} b_2^{Q/(Q-1)}
	\le \eps b_1^Q + \eps^{-1/(Q-1)} b_2^{Q/(Q-1)}.
\end{equation}
Consider $j\in\N$ large enough that $1/j< \beta$. We estimate
\begin{align*}
		&2j\sum_{k}\widehat{L}_{j,k} \mu(2\widehat{B}_{j,k})\\
		&\ \  \le 2j\sum_{k}\widehat{l}_{j,k}\mu(2\widehat{B}_{j,k})
		\left(\frac{\kappa(M\widehat{B}_{j,k})}{\mu(M\widehat{B}_{j,k})}\right)^{(Q-1)/Q}
		(H_f^{\kappa,M}(y_{j,k})+\eps)\quad\textrm{by }\eqref{eq:Dj def}\\
		&\ \  \le 2C_d  C_{\Om}^{1/Q} \sum_{k} \widehat{l}_{j,k} \kappa(M\widehat{B}_{j,k})^{(Q-1)/Q} (H_f^{\kappa,M}(y_{j,k})+\eps)
		\quad\textrm{by }\eqref{eq:measure lower and upper bound}\\
		&\ \  \le (2C_d)^Q C_{\Om}\eps \sum_{k}\widehat{l}_{j,k}^Q
		+\eps^{-1/(Q-1)}\sum_{k} (H_f^{\kappa,M}(y_{j,k})+\eps)^{Q/(Q-1)}\kappa(M\widehat{B}_{j,k})\quad\textrm{by }
		\eqref{eq:generalized Young}\\
		&\ \  \le (2C_d)^Q C_{\Om}^2\eps  \sum_{k}\nu(B(f(y_{j,k}),\widehat{l}_{j,k}))
		\quad \textrm{by }\eqref{eq:measure lower and upper bound}
		\\
		&\qquad \qquad +2^{Q/(Q-1)}\eps^{-1/(Q-1)}\sum_{k} \left(H_f^{\kappa,M}(y_{j,k})^{Q/(Q-1)}+\eps^{Q/(Q-1)}\right)\kappa(M\widehat{B}_{j,k})\\
		&\ \ \le (2C_d)^Q C_{\Om}^2 C_M \eps  \nu(V)
		+C_M 2^{Q/(Q-1)} \eps^{-1/(Q-1)}\int_{\Om}(h+2\eps^{Q/(Q-1)})\,d\kappa
	\end{align*}
by \eqref{eq:finite nu measure estimate} and \eqref{eq:Dj def 2}.
By combining this with \eqref{eq:A1 estimate}, we get
\begin{equation}\label{eq:g energy estimate}
	\begin{split}
		\int_{\Om}g_j\,d\mu
		& = 2j\sum_{k} L_{k} \mu(2B_{j,k})
		+2j\sum_{j,k}\widehat{L}_{j,k} \mu(2\widehat{B}_{j,k})\\
		&  \le   2^{2+Q/(Q-1)} C_d C_M \eps^{-1/(Q-1)} \int_{\Om}(h+2\eps^{Q/(Q-1)})\,d\kappa
		+(2C_d)^Q C_{\Om}^2 C_M \eps \nu(V)\\
		&  \le   C_1 \eps^{-1/(Q-1)} \int_{\Om}h\,d\kappa+ C_1\eps\kappa(\Om)+C_1 \eps \nu(V)\\
		& <\infty
	\end{split}
\end{equation}
by the assumption \eqref{eq:A1A2 assumption}, with
$C_1:=2^{3+Q+Q/(Q-1)} C_{\Om}^2 C_d^{Q+\lceil \log_2 (18 M)\rceil}$.
By \eqref{eq:ug estimate}, for $1$-a.e. curve $\gamma\colon [0,\ell_{\gamma}]\to \Om$ we have
\[
d_Y(f(\gamma(0)),f(\gamma(\ell_{\gamma})))\le
\liminf_{j\to\infty}\int_{\gamma}g_j\,ds.
\]
By the properties of modulus,
see e.g. \cite[Lemma 1.34]{BB},
it follows that for $1$-a.e. curve $\gamma\colon [0,\ell_{\gamma}]\to \Om$, we have in fact
\[
d_Y(f(\gamma(t_1)),f(\gamma(t_2)))
\le \liminf_{j\to\infty}\int_{\gamma|_{[t_1,t_2]}}g_j\,ds
\]
for every $0\le t_1<t_2\le \ell_{\gamma}$.
Thus $f\in D^{\BV}(\Om;Y)$ by Definition \ref{def:BV def},
and by \eqref{eq:g energy estimate} we have \eqref{eq:BV regularity claim}.\\
\end{proof}

In the remaining two theorems, we show that if $d\kappa=a\,d\mu$
and either $\Lip_f^{a,M}$ or $H_f^{a,M}$ is not too large,
then we can get Newton-Sobolev instead of just BV regularity.

\begin{theorem}\label{thm:main theorem 2}
	Suppose $\Om$ is nonempty, open, and bounded,
	$f\colon\Om\to Y$,
	$M\ge 1$, $1\le p<\infty$,
	$a\in L^p(\Om)$ is a nonnegative function whose integral over every
	nonempty open subset of $\Om$ is nonzero,
	and suppose that $\Om$ admits a partition into disjoint sets $A$ and $N$ such that
	$\Lip_f^{a,M}<\infty$ in $A$,
	$\Mod_p(\Gamma_N)=0$, and
	\begin{equation}\label{eq:hap assumption}
	\int_\Om (ha)^p\,d\mu<\infty\quad\textrm{for a function }h\ge \Lip_f^{a,M}\ch_A.
	\end{equation}
	Furthermore, suppose that for some $0<C_{\Om}<\infty$ and for all $r>0$, we have
	\begin{equation}\label{eq:measure lower and upper bound 2}
		\frac{\mu(B(x,r))}{r^Q}\le C_{\Om}
		\ \textrm{ for all }x\in \Om.
	\end{equation}
	Then $f\in D^{p}(\Om;Y)$ with
	\begin{equation}\label{eq:C2 energy estimate}
		\Vert f\Vert_{D^p(\Om;Y)}\le C_2\Vert ha \Vert_{L^p(\Om)},
	\end{equation}
	where $C_2$ only depends on $C_d$, $p$, and $M$.
\end{theorem}

\begin{proof}
	Fix $0<\eps\le 1$.
	By the Vitali-Carath\'eodory theorem (Theorem \ref{thm:VitaliCar}), we can assume that $h$ is
	lower semicontinuous.
	We have $A=\bigcup_{j=1}^{\infty}A_j$, where $A_j$ consists of those points $x\in A\cap \Om(M/j)$ for which
	\begin{equation}\label{eq:Aj def with a}
		\sup_{0<r\le 1/j}
		\frac{L_f(x,r)}{r}\left(\vint{B(x,Mr)}a\,d\mu\right)^{-1}
		\le \Lip_f^{a,M}(x)+\eps
	\end{equation}
	and
	\begin{equation}\label{eq:Aj def 2 with a}
		\Lip_f^{a,M}(x)\le h(z) +\eps
		\quad\textrm{for all }z\in B(x,M/j).
	\end{equation}
	For any fixed $j\in\N$, we find a collection $\{B_{j,k}=B(x_{j,k},1/j)\}_{k}$ covering $A_j$ with $x_{j,k}\in A_j$,
	and such that the balls $(M+1)B_{j,k}$ can be divided into at most $C_{M}$
	collections of pairwise disjoint balls.
	
	Denote $L_{j,k}:=L_f(x_{j,k},1/j)$.
	For each $j\in\N$, define
	\[
	g_j:=2j\sum_{k}L_{j,k}\ch_{2B_{j,k}}.
	\]
	By assumption, $p$-a.e nonconstant curve in $\Om$ avoids $N$.
	Hence, just as in the estimates leading to \eqref{eq:ug estimate}, for $p$-a.e. curve $\gamma$ in $\Om$
	we get
		\begin{equation}\label{eq:ug estimate in lip case}
			d_Y(f(\gamma(0)),f(\gamma(\ell_{\gamma})))
			\le \liminf_{j\to\infty}\int_\gamma g_j\,ds.
	\end{equation}
	Then we estimate
	\begin{equation}\label{eq:A1 estimate }
		\begin{split}
			j^p\sum_{k} L_{j,k}^p\mu(2B_{j,k})
			& \le  \sum_{k}\mu(2B_{j,k})\left(\vint{MB_{j,k}}a\,d\mu\right)^p(\Lip_f^{a,M}(x_{j,k})+\eps)^p\quad\textrm{by }\eqref{eq:Aj def with a}\\
			& \le  \sum_{k}\mu(2B_{j,k})\left(\vint{MB_{j,k}}(h+2\eps)a\,d\mu\right)^p\quad\textrm{by }\eqref{eq:Aj def 2 with a}\\
			& \le  \sum_{k}\mu(2B_{j,k})
			\vint{MB_{j,k}}(h+2\eps)^pa^p\,d\mu\quad\textrm{by H\"older}\\
			& \le  C_d
				\sum_{k}\int_{MB_{j,k}}(h+2\eps)^pa^p\,d\mu\\
			& \le  C_d C_M\int_{\Om}(h+2\eps)^pa^p\,d\mu.
		\end{split}
	\end{equation}
	Thus for all $j\in\N$, we get
	\begin{equation}\label{eq:g energy estimate p}
		\begin{split}
			\int_{\Om}g_j^p\,d\mu
			&\le \int_{\Om}\left(2j\sum_{k} L_{j,k}\ch_{2B_{j,k}}\right)^p\,d\mu\\
			&\le 2^pC_M^p j^p \sum_{k} L_{j,k}^p\mu(2B_{j,k})\\
			&\le 2^p C_d C_M^{p+1} \int_{\Om}(h+2\eps)^p a^p\,d\mu\quad\textrm{by }\eqref{eq:A1 estimate }\\
			&\le 2^{2p} C_d C_M^{p+1} \left(\int_{\Om}h^p a^p\,d\mu+(2\eps)^p\int_{\Om}a^p \,d\mu\right)\quad\\
			&<\infty\quad\textrm{by }\eqref{eq:hap assumption}.
		\end{split}
	\end{equation}
	Hence $\{g_j\}_{j=1}^{\infty}$ is a bounded sequence in $L^p(\Om)$.
	
		\textbf{Case $p=1$:}
		
		Recall that we defined
		\[
		g_j= 2j\sum_{k} L_{j,k} \ch_{2B_{j,k}},\quad j\in\N.
		\]
		By \eqref{eq:g energy estimate p}, this is a bounded sequence in $L^1(\Om)$.
		We will show that it is equi-integrable.
		The first condition of Definition \ref{def:equiintegrability}
		holds automatically since $\Om$ as a bounded set has finite $\mu$-measure.
		We check the second condition.
		
		Suppose by contradiction that
		we find $0<\alpha<1$, an infinite set $J\subset \N$, and
		a sequence of $\mu$-measurable sets $H_j\subset \Om$ such that $\mu(H_j)\to 0$ and
		\begin{equation}\label{eq:equiintegrability contradiction}
			j\int_{H_j} \sum_{k} L_{j,k}\ch_{2B_{j,k}}\,d\mu\ge \alpha\quad\textrm{for all }j\in J.
		\end{equation}
		Choose $L$ to be the following (very large) number:
		\begin{equation}\label{eq:M choice }
			L:=\frac{2 C_d^{1+\lceil 
					\log_2 M\rceil} C_M}{\alpha} \int_{\Om}(h+2)a\,d\mu.
		\end{equation}
		For every $j\in\N$, define two sets of indices $I_1^{j}$ and $I_2^{j}$ as follows:
		for $k\in I_1^{j}$, we have
		\[
		\frac{\mu((M+1)B_{j,k}\cap H_j)}{\mu((M+1)B_{j,k})}\le \frac{1}{L},
		\]
		and $I_2^{j}$ consists of the remaining indices.
		We estimate
		\begin{equation}\label{eq:I1i estimate}
			\begin{split}
				 j\sum_{k\in I^j_1} L_{j,k}\mu(2B_{j,k}\cap H_j)
				&\le j\sum_{k\in I^j_1} L_{j,k} \mu((M+1)B_{j,k}\cap H_j)\\
				&\le \frac{j}{L}\sum_{k\in I^j_1}L_{j,k}\mu((M+1)B_{j,k})\\
				&\le \frac{C_d^{\lceil 
						\log_2 M\rceil}j}{L}\sum_{k\in I^j_1}L_{j,k}\mu(2B_{j,k})\\
				&\le \frac{C_d^{1+\lceil 
						\log_2 M\rceil} C_M}{L} \int_{\Om}(h+2)a\,d\mu\quad\textrm{by }\eqref{eq:A1 estimate }\\
				&= \frac{\alpha}{2}\quad\textrm{by }\eqref{eq:M choice }.
			\end{split}
		\end{equation}
		Next, we estimate
		\begin{equation}\label{eq:I2 measure estimate}
			\begin{split}
			\mu\left(\bigcup_{k\in I_2^j}MB_{j,k}\right)
			&\le \sum_{k\in I_2^j}\mu((M+1)B_{j,k})\\
			&\le   L\sum_{k\in I_2^j}\mu((M+1)B_{j,k}\cap H_j)
			\le  L  C_M \mu(H_j)
			\to 0\quad\textrm{as }j\to\infty.
			\end{split}
		\end{equation}
		By writing the first four lines of \eqref{eq:A1 estimate }
		with the sums over the indices $k\in I_2^j$, we get
		\begin{align*}
			j\sum_{k\in I_2^j} L_{j,k}\mu(2B_{j,k})
			& \le  C_d  \sum_{k\in I_2^j}\int_{MB_{j,k}}(h+2)a\,d\mu\\
			& \le  C_d C_M \int_{\bigcup_{k\in I_2^j}MB_{j,k}}(h+2)a\,d\mu\\
			& \to 0\quad\textrm{as }j\to\infty
		\end{align*}
		by \eqref{eq:I2 measure estimate} and by the absolute continuity of integrals.
		Combining this with \eqref{eq:I1i estimate}, we get
		\[
		\limsup_{j\to\infty}\int_{H_j}j\sum_{k} L_{j,k}\ch_{2B_{j,k}}\,d\mu\le \frac{\alpha}{2}.
		\]
		This contradicts \eqref{eq:equiintegrability contradiction} and proves the equi-integrability
		of $\{g_j\}_{j=1}^{\infty}$.
		Now by the Dunford--Pettis theorem (Theorem \ref{thm:dunford-pettis}), we find $g\in L^1(\Om)$
		such that by passing to a subsequence (not relabeled), we have
		$g_j\to g$ weakly in $L^1(\Om)$.
		By Mazur's lemma (Theorem \ref{thm:Mazur lemma}), for suitable convex combinations we get the
		strong convergence $\sum_{l=j}^{N_j}a_{j,l} g_l\to g$ in $L^1(\Om)$.
		From \eqref{eq:ug estimate in lip case} and Fuglede's lemma (Lemma \ref{lem:Fuglede lemma}), we get
		\[
				d_Y(f(\gamma(0)),f(\gamma(\ell_{\gamma})))
				\le \liminf_{j\to\infty}\int_{\gamma}\sum_{l=j}^{N_j}a_{j,l} g_l\,ds
				=\int_{\gamma}g \,ds
		\]
		for $1$-a.e. curve $\gamma$ in $\Om$.
		Thus $g\in L^1(\Om)$ is a $1$-weak upper gradient of $f$ in $\Om$, and so $f\in D^1(\Om;Y)$.
		
		\textbf{Case $1<p<\infty$:}
		
		By \eqref{eq:g energy estimate p}, $\{g_j\}_{j=1}^{\infty}$ is a bounded sequence in $L^p(\Om)$.
		By reflexivity of the space $L^p(\Om)$, we find a subsequence of $\{g_j\}_{j=1}^{\infty}$ (not relabeled)
		and $g\in L^p(\Om)$ such that $g_j\to g$ weakly in $L^p(\Om)$ (see e.g. \cite[Section 2]{HKSTbook}).
		By Mazur's lemma (Theorem \ref{thm:Mazur lemma}), for suitable convex combinations we get the
		strong convergence $\sum_{l=j}^{N_j}a_{j,l} g_l\to g$ in $L^p(\Om)$.
		By \eqref{eq:ug estimate in lip case} and Fuglede's lemma (Lemma \ref{lem:Fuglede lemma}),
		we obtain for $p$-a.e. curve $\gamma$ in $\Om$ that
		\[
		d_Y(f(\gamma(0)),f(\gamma(\ell_{\gamma})))
		\le \liminf_{j\to\infty}\int_{\gamma}\sum_{l=j}^{N_j}a_{j,l} g_l\,ds
		= \int_{\gamma}g\,ds.
		\]
		Hence $f\in D^p(\Om;Y)$.
		
		For all $1\le p<\infty$, by \eqref{eq:g energy estimate p} we get
		\[
		\Vert f\Vert_{D^p(\Om;Y)}\le \Vert g\Vert_{L^p(\Om)}
		\le\limsup_{j\to\infty}\Vert g_j\Vert_{L^p(\Om)}
		\le C_2(\Vert h a\Vert_{L^{p}(\Om)}^p+(2\eps)^2\Vert a\Vert_{L^p(\Om)}^p)^{1/p}
		\]
		with $C_2=4 C_d^{1/p+(1+1/p)\lceil \log_2(18M)\rceil}$. Letting $\eps\to 0$, this
		proves \eqref{eq:C2 energy estimate}.\\
	\end{proof}

Recall that $1<Q<\infty$ is a fixed parameter,
and that for $1\le p\le Q$,
we denote the Sobolev conjugate by $p^*=Qp/(Q-p)$ when $p<Q$,
and $p^*=\infty$ when $p=Q$.

\begin{theorem}\label{thm:main theorem 3}
	Suppose $\Om$ is nonempty, open, and bounded,
	$f\colon\Om\to Y$ is injective, $M\ge 1$, $1\le p\le Q$,
	$a\in  L^{p^*(Q-1)/Q}(\Om)$ is a nonnegative function whose integral over every
	nonempty open subset of $\Om$ is nonzero,
	and suppose $\Om$ admits a partition into disjoint sets $D$ and $N$ such that
	$H_f^{a,M}<\infty$ in $D$,
	$\mu(N)=\Mod_p(\Gamma_N)=0$, and
	\begin{equation}\label{eq:p star assumption}
	\Vert h a^{p^*(Q-1)/Q}\Vert_{L^{1}(\Om)}<\infty
	\quad\textrm{for a function }h\ge (H_f^{a,M})^{p^*}\ch_D
	\end{equation}
	in the case $1\le p<Q$, and $ \Vert H_f^{a,M}\Vert_{L^{\infty}(\Om)}<\infty$
	in the case $p=Q$.

	Furthermore, suppose that for some $0<C_{\Om}<\infty$ and for all $r>0$,
	\begin{equation}\label{eq:measure lower and upper bound 3}
		\frac{\mu(B(x,r))}{r^Q}\le C_{\Om}
		\ \textrm{ for all }x\in \Om
		\quad\textrm{and}\quad
		\frac{\nu(B(f(x),r))}{r^Q}\ge C_{\Om}^{-1}
		\ \textrm{ for all }x\in D,
	\end{equation}
	and that there is $\beta>0$ such that
	\begin{equation}\label{eq:finite image assumption 3}
		\nu(V)<\infty\quad\textrm{for }V:=\bigcup_{y\in D}B(f(y),l_f(y,\beta)).
	\end{equation}
	Then $f\in D^{p}(\Om;Y)$ with
	\begin{equation}\label{eq:C3 energy estimate}
		\Vert f\Vert_{D^p(\Om;Y)}^p \le C_3\nu(V)+C_3\Vert h a^{p(Q-1)/(Q-p)}\Vert_{L^{1}(\Om)}
	\end{equation}
	in the case $1\le p<Q$,
	and
	\begin{equation}\label{eq:C3 energy estimate Q}
		\Vert f\Vert_{D^Q(\Om;Y)}^Q \le C_4 \nu(V) \Vert a\Vert_{L^{\infty}(\Om)}^{Q-1}
		\Vert H_f^{a,M}\Vert_{L^{\infty}(\Om)}^Q
	\end{equation}
	in the case $p=Q$.
	The constants $C_3,C_4$ only depend on $C_d$, $C_{\Om}$, $p$, $Q$, and $M$.
\end{theorem}
\begin{proof}
	\textbf{Part 1: Case $1\le p<Q$.}\\
	Fix $0<\eps\le 1$.
	By the Vitali-Carath\'eodory theorem (Theorem \ref{thm:VitaliCar}),
	we can assume that $h$ is lower semicontinuous.
	We have $D=\bigcup_{j=1}^{\infty}D_j$, where $D_j$ consists of those points
	$y\in D\cap \Om((M+1)/j)$ for which
	\begin{equation}\label{eq:Dj def new}
	\sup_{0<s\le 1/j}\frac{L_f(y,s)}{l_f(y,s)}
	\left(\vint{B(y,Ms)}a\,d\mu\right)^{-(Q-1)/Q}
	\le H_f^{a,M}(y)+\eps
	\end{equation}
	and
	\begin{equation}\label{eq:Dj def new 2}
		(H_{f}^{a,M}(y))^{Qp/(Q-p)}
		\le h(z)+\eps\quad\textrm{for all }z\in B(y,M/j).
	\end{equation}
	As before, we find a collection $\{B_{j,k}=B(y_{j,k},1/j)\}_{k}$
	of balls covering $D_j$, with $y_{j,k}\in D_j$,
	and such that the balls $(M+1)B_{j,k}$ can be divided into at most $C_{M}$
	collections of pairwise disjoint balls.

	Denote $L_{j,k}:=L_f(y_{j,k},1/j)$ and $l_{j,k}:=l_f(y_{j,k},1/j)$. For each $j\in\N$, define
\[
g_j:=2j\sum_{k} L_{j,k}\ch_{2B_{j,k}}.
\]
Just as in the estimates leading to \eqref{eq:ug estimate}, for $p$-a.e. curve $\gamma$ in $\Om$ we get
\begin{equation}\label{eq:ug estimate in qc case}
	d_Y(f(\gamma(0)),f(\gamma(\ell_{\gamma})))
	\le \liminf_{j\to\infty}\int_\gamma g_j\,ds.
\end{equation}
Then we estimate
	\begin{equation}\label{eq:A2 estimate p}
		\begin{split}
			&j^p\sum_{k}L_{j,k}^p \mu(2B_{j,k})\\
			&\ \ \le j^p\sum_{k} l_{j,k}^p\mu(2B_{j,k})
			\left(\vint{MB_{j,k}}a\,d\mu\right)^{p(Q-1)/Q}
			(H_f^{a,M}(y_{j,k})+\eps)^p\quad\textrm{by }\eqref{eq:Dj def new}\\
			&\ \ \le j^p\sum_{k}l_{j,k}^p \mu(2B_{j,k})
			\left(\vint{M B_{j,k}}a^{p(Q-1)/(Q-p)}\,d\mu\right)^{(Q-p)/Q} (H_f^{a,M}(y_{j,k})
			+\eps)^p \quad\textrm{by H\"older}\\
			&\ \ \le C_d C_{\Om}^{p/Q}
			\sum_{k}l_{j,k}^p \left(\int_{M B_{j,k}}a^{p(Q-1)/(Q-p)}\,d\mu\right)^{(Q-p)/Q} (H_f^{a,M}(y_{j,k})+\eps)^p
		\end{split}
	\end{equation}
	by the doubling property of $\mu$ and \eqref{eq:measure lower and upper bound 3}.
	Using this and Young's inequality, we get
	\begin{equation}\label{eq:qc energy estimate p}
		\begin{split}
			&j^p\sum_{k}L_{j,k}^p \mu(2B_{j,k})\\
			& \le (C_d C_{\Om}^{p/Q})^{Q/p} \sum_{k}l_{j,k}^Q
			 +\sum_{k} \left(\int_{MB_{j,k}}a^{p(Q-1)/(Q-p)}\,d\mu\right) (H_f^{a,M}(y_{j,k})+\eps)^{Qp/(Q-p)}\\
			& \le C_d^{Q/p} C_{\Om} \sum_{k}l_{j,k}^Q\\
			&\qquad +2^{Qp/(Q-p)}\sum_{k} \Bigg(\int_{MB_{j,k}}a^{p(Q-1)/(Q-p)}\,d\mu\Bigg) \Big(\big(H_f^{a,M}(y_{j,k})\big)^{Qp/(Q-p)}+\eps^{Qp/(Q-p)}\Big)\\
			& \le  C_d^{Q/p} C_{\Om}^2 \sum_{k}\nu(B(f(y_{j,k}),l_{j,k}))
			\quad\textrm{by }\eqref{eq:measure lower and upper bound 3}\\
			&\qquad +2^{Qp/(Q-p)}\sum_{k} \int_{MB_{j,k}}a^{p(Q-1)/(Q-p)}\big(h+
			\eps+\eps^{Qp/(Q-p)}\big)\,d\mu\quad\textrm{by }\eqref{eq:Dj def new 2}\\
			& \le  C_d^{Q/p} C_{\Om}^2C_M\nu(V)+ 2^{Qp/(Q-p)}C_M\int_{\Om}a^{p(Q-1)/(Q-p)}(h+2\eps)\,d\mu
		\end{split}
	\end{equation}
	for all $j\in\N$ large enough that $1/j< \beta$,
	 by \eqref{eq:finite nu measure estimate}; note that the exact same estimate applies here.
	We get
	\begin{equation}\label{eq:g star energy estimate p}
		\begin{split}
			\int_{\Om}g_j^p\,d\mu
			&= \int_{\Om}\left(2j\sum_{k} L_{j,k}\ch_{2B_{j,k}}\right)^p\,d\mu\\
			&\le (2j)^pC_M^p \sum_{k}  L_{j,k}^p \mu(2B_{j,k})\\
			&\le C_3 \nu(V)+  C_3 \int_{\Om}a^{p(Q-1)/(Q-p)}(h+2\eps)\,d\mu
			\quad\textrm{by }\eqref{eq:qc energy estimate p}\\
			& <\infty\quad\textrm{by }\eqref{eq:p star assumption}
		\end{split}
	\end{equation}
	with $C_3:=2^{p+Qp/(Q-p)} C_d^{Q/p+(p+1)\lceil \log_2 (18M)\rceil } C_{\Om}^2$.
	
	\textbf{Part 1(a): Case $p=1$.}\\
	Recall that we are considering the sequence
	\[
	g_j=2j\sum_{k} L_{j,k}\ch_{2B_{j,k}}.
	\]
	By \eqref{eq:g star energy estimate p}, this is a bounded sequence in $L^1(\Om)$.
	We will show that it is equi-integrable.
	The first condition of Definition \ref{def:equiintegrability}
	holds automatically since $\Om$ as a bounded set has finite $\mu$-measure.
	We check the second condition.
	
	Suppose by contradiction that
	we find $0<\alpha<1$, an infinite set $J\subset \N$, and
	a sequence of $\mu$-measurable sets $H_j\subset \Om$ such that $\mu(H_j)\to 0$ and
	\begin{equation}\label{eq:equiintegrability contradiction 2}
		j\int_{H_j} \sum_{k} L_{j,k}\ch_{2B_{j,k}}\,d\mu
		\ge \alpha\quad\textrm{for all }j\in J.
	\end{equation}
	Choose $L$ to be the following (very large) number:
	\begin{equation}\label{eq:M choice 3}
		L:=\frac{4C_d^{\lceil \log_2 M\rceil}}{\alpha}\Bigg[C_d^{Q}C_{\Om}^2C_M\nu(V)
		+2^{Q/(Q-1)}C_M\int_{\Om}a(h+2)\,d\mu\Bigg].
	\end{equation}
	For every $j\in \N$, define two sets of indices $I_1^{j}$ and $I_2^{j}$ as follows:
	for $k\in I_1^{j}$, we have
	\[
	\frac{\mu((M+1)B_{j,k}\cap H_j)}{\mu((M+1)B_{j,k})}
	\le \frac{1}{L},
	\]
	and $I_2^{j}$ consists of the remaining indices.
	Now
	\begin{equation}\label{eq:I1 estimate}
		\begin{split}
			&j\sum_{k\in I^j_1}L_{j,k}\mu(2B_{j,k}\cap H_j)\\
			&\qquad \le j\sum_{k\in I^j_1}L_{j,k}\mu((M+1)B_{j,k}\cap H_j)\\
			&\qquad \le \frac{j}{L}\sum_{k\in I^j_1}L_{j,k}\mu((M+1)B_{j,k})\\
			&\qquad \le \frac{j}{L}C_d^{\lceil \log_2 M\rceil}\sum_{k\in I^j_1}L_{j,k}\mu(2B_{j,k})\\
			&\qquad \le \frac{C_d^{\lceil \log_2 M\rceil}}{L} \Bigg[C_d^{Q}C_{\Om}^2C_M\nu(V)
			+2^{Q/(Q-1)}C_M\int_{\Om}a(h+2)\,d\mu\Bigg]
			\quad\textrm{by }\eqref{eq:qc energy estimate p}\\
			&\qquad = \frac{\alpha}{4}\quad\textrm{by }\eqref{eq:M choice 3}.
		\end{split}
	\end{equation}
	Next, we estimate
	\begin{equation}\label{eq:I2 measure estimate hat}
		\begin{split}
		\mu\left(\bigcup_{k\in I_2^j}MB_{j,k}\right)
		&\le \sum_{k\in I_2^j}\mu((M+1)B_{j,k})\\
		&\le  L\sum_{k\in I_2^j}\mu((M+1)B_{j,k}\cap H_j)
		\le  L  C_M \mu(H_j)
		\to 0\quad\textrm{as }j\to\infty.
		\end{split}
	\end{equation}
	Just as in \eqref{eq:generalized Young}, for any $0\le \delta<1$ and any $b_1,b_2\ge 0$ we have
	\begin{equation}\label{eq:generalized Young 2}
		b_1 b_2 \le \delta b_1^Q + \delta^{-1/(Q-1)} b_2^{Q/(Q-1)}.
	\end{equation}
	We choose
	\begin{equation}\label{eq:delta definition}
		\delta:=\frac{1}{4}\frac{\alpha}{C_d^Q C_{\Om}^2 C_M \nu(V)+1}.
	\end{equation}
	By \eqref{eq:Dj def new}, for $j\in \N$ large enough that $1/j<\beta$ we get
	\begin{align*}
		&j\sum_{k\in I^j_2}L_{j,k}\mu(2B_{j,k})\\
		&\quad \le j\sum_{k\in I^j_2}l_{j,k}\mu(2B_{j,k})\left(
		\vint{MB_{j,k}}a\,d\mu\right)^{(Q-1)/Q}(H_f^{a,M}(y_{j,k})+1)\\
		&\quad \le C_d C_{\Om}^{1/Q} \sum_{k\in I^j_2}l_{j,k}\left(
		\int_{MB_{j,k}}a\,d\mu\right)^{(Q-1)/Q}(H_f^{a,M}(y_{j,k})+1)
		\quad\textrm{by }\eqref{eq:measure lower and upper bound 3}\\
		&\quad \le C_d^Q C_{\Om}\delta  \sum_{k\in I^j_2}l_{j,k}^Q + \delta^{-1/(Q-1)}\sum_{k\in I^j_2} (H_f^{a,M}(y_{j,k})+1)^{Q/(Q-1)}
		\int_{MB_{j,k}}a\,d\mu\quad\textrm{by }\eqref{eq:generalized Young 2}\\
		&\quad \le C_d^Q C_{\Om}^2 \delta  \sum_{k\in I^j_2}\nu(B(f(y_{j,k}),l_{j,k}))
		\quad\textrm{by }\eqref{eq:measure lower and upper bound 3}\\ 
		&\qquad\qquad + 2^{Q/(Q-1)}\delta^{-1/(Q-1)}\sum_{k\in I^j_2} \big(H_f^{a,M}(y_{j,k})^{Q/(Q-1)}+1\big)\int_{MB_{j,k}}a\,d\mu\\
		&\quad \le C_d^Q C_{\Om}^2 C_M \delta \nu(V)
		\quad\textrm{by }\eqref{eq:finite nu measure estimate}\textrm{ (the same estimate applies)}\\
		 &\qquad\qquad +  2^{Q/(Q-1)}\delta^{-1/(Q-1)}\sum_{k\in I^j_2} \int_{MB_{j,k}}(h+2)a\,d\mu\quad\textrm{by }\eqref{eq:Dj def new 2}\\
		&\quad \le \frac{\alpha}{4}
		+ 2^{Q/(Q-1)} C_{M}\delta^{-1/(Q-1)} \int_{\bigcup_{k\in I_2^j}MB_{j,k}}(h+2)a\,d\mu\quad\textrm{by }\eqref{eq:delta definition}.
	\end{align*}
	Here the latter term goes to zero as $j\to\infty$ by \eqref{eq:I2 measure estimate hat}.
	Combining this with \eqref{eq:I1 estimate}, we get
	\[
	\limsup_{j\to\infty}j
	\sum_{k}L_{j,k}\mu(2B_{j,k}\cap H_j)
	\le \frac{\alpha}{2}.
	\]
	This contradicts \eqref{eq:equiintegrability contradiction 2} and proves the equi-integrability of
	$\{g_j\}_{j=1}^{\infty}$.
	
	Now by the Dunford--Pettis theorem (Theorem \ref{thm:dunford-pettis}), we find $g\in L^1(\Om)$
	such that by passing to a subsequence (not relabeled), we have
	$g_j\to g$ weakly in $L^1(\Om)$.
	By Mazur's lemma (Theorem \ref{thm:Mazur lemma}), for suitable convex combinations we get the
	strong convergence $\sum_{l=j}^{N_j}a_{j,l} g_l\to g$ in $L^1(\Om)$.
	From \eqref{eq:ug estimate in qc case} and Fuglede's lemma (Lemma \ref{lem:Fuglede lemma}) we get
	\[
				d_Y(f(\gamma(0)),f(\gamma(\ell_{\gamma})))
			\le \liminf_{j\to\infty}\int_{\gamma}\sum_{l=j}^{N_j}a_{j,l}  g_l\,ds
			=\int_{\gamma} g \,ds
	\]
	for $1$-a.e. curve $\gamma$ in $\Om$.
	Thus $g\in L^1(\Om)$ is a $1$-weak upper gradient of $f$  in $\Om$, and so $f\in D^1(\Om;Y)$.
	
	\textbf{Part 1(b): Case $1<p<Q$.}\\
	By \eqref{eq:g star energy estimate p}, 
	$\{g_j\}_{j=1}^{\infty}$ is a bounded sequence in $L^p(\Om)$.
	By reflexivity of the space $L^p(\Om)$, we find a subsequence of $\{g_j\}_{j=1}^{\infty}$ (not relabeled)
	and $g\in L^p(\Om)$ such that $g_j\to g$ weakly in $L^p(\Om)$.
	By Mazur's lemma (Theorem \ref{thm:Mazur lemma}), for suitable convex combinations we get the
	strong convergence $\sum_{l=j}^{N_j}a_{j,l} g_l\to g$ in $L^p(\Om)$.
	By \eqref{eq:ug estimate in qc case} and Fuglede's lemma (Lemma \ref{lem:Fuglede lemma}),
	we obtain
	\begin{equation}
		\begin{split}
			d_Y(f(\gamma(0)),f(\gamma(\ell_{\gamma})))
			\le \liminf_{j\to\infty}\int_{\gamma}\sum_{l=j}^{N_j}a_{j,l} g_l\,ds
			= \int_{\gamma}g\,ds
		\end{split}
	\end{equation}
	for $p$-a.e. curve $\gamma$ in $\Om$.
	Hence $f\in D^p(\Om;Y)$.
	For all $1\le p<Q$, 
	by \eqref{eq:g star energy estimate p} we get
	\[
	\Vert f\Vert_{D^p(\Om;Y)}^p
	\le \Vert g\Vert_{L^p(\Om)}^p
	\le \limsup_{j\to\infty}\Vert g_j\Vert_{L^p(\Om)}^p
	\le C_3\nu(V)+C_3\Vert (h+2\eps) a^{p(Q-1)/(Q-p)}\Vert_{L^{1}(\Om)}.
	\]
	Letting $\eps\to 0$, we get \eqref{eq:C3 energy estimate}.
	
	\textbf{Part 2: Case $p=Q$.}\\
Now $\Om$ is the disjoint union of $N$ and $D$, where $\mu(N)=\Mod_Q(\Gamma_N)=0$, and
$H_f^{a,M}<\infty$ in $D$ and
$\Vert H_f^{a,M}\Vert_{L^{\infty}(\Om)}<\infty$, where $a\in L^{\infty}(\Om)$.
For every $y\in \Om$, we have
\begin{align*}
H_f(y)
&=\limsup_{r\to 0}\frac{L_f(y,r)}{l_f(y,r)}\\
&\le \Vert a\Vert_{L^{\infty}(\Om)}^{(Q-1)/Q}
\limsup_{r\to 0}\frac{L_f(y,r)}{l_f(y,r)}\left(\vint{B(y,Mr)}a\,d\mu\right)^{-(Q-1)/Q}\\
&= \Vert a\Vert_{L^{\infty}(\Om)}^{(Q-1)/Q} H_f^{a,M}(y).
\end{align*}
In particular,
\begin{equation}\label{eq:L infinity comparison}
\Vert H_f \Vert_{L^{\infty}(\Om)}
\le \Vert a\Vert_{L^{\infty}(\Om)}^{(Q-1)/Q} \Vert H_f^{a,M}\Vert_{L^{\infty}(\Om)}<\infty.
\end{equation}
Thus in fact $H_f<\infty$ in $D$, and $H_f\in L^{\infty}(\Om)$.
(Hence from now on, we can work with the ordinary distortion number $H_f$.)
We have $D=\bigcup_{j=1}^{\infty}D^j$, where $D^j$ consists of those points
$y\in D\cap \Om(2/j)$ for which
\begin{equation}\label{eq:finiteness j condition}
\sup_{0<s\le 1/j}\frac{L_f(y,s)}{l_f(y,s)} \le j.
\end{equation}
Fix $j\in\N$.
For each $k\ge j$,
we find a collection $\{B_{k,l}=B(y_{k,l},1/k)\}_{l}$ covering $D^j$, with
$y_{k,l}\in D^j$,
and such that the balls $2B_{k,l}$ can be divided into at most $C_{M}$
collections of pairwise disjoint balls.
Define
\[
g_k:=2k\sum_{l}L_f(y_{k,l},1/k)\ch_{2B_{k,l}},\quad k\in\N.
\]
Consider a curve $\gamma$ in $\Om$ with $\ell_{\gamma}\ge 1/k$.
If $\gamma$ intersects $B_{k,l}$, then $\mathcal H^1(\gamma\cap 2B_{k,l})\ge 1/k$.
Thus we have
\begin{equation}\label{eq:upper gradient property proved}
	\int_{\gamma}g_k\,ds
	\ge 2\sum_{l\colon \gamma\cap B_{k,l}\neq\emptyset}L_f(y_{k,l},1/k)
	\ge  \sum_{l\colon \gamma\cap B_{k,l}\neq\emptyset}\diam f(B_{k,l})
	\ge \mathcal H_{\infty}^{1}(f(\gamma\cap D^j)),
\end{equation}
where the last inequality holds since the balls $B_{k,l}$ satisfying
$\gamma\cap B_{k,l}\neq\emptyset$
cover $\gamma\cap D^j$ and so the sets $f(B_{k,l})$ with
$\gamma\cap B_{k,l}\neq\emptyset$ cover $f(\gamma\cap D^j)$.

Denote $L_{k,l}:=L_f(y_{k,l},1/k)$ and $l_{k,l}:=l_f(y_{k,l},1/k)$.
For every $k\in\N$ such that $1/k<\beta$, we get
\begin{align*}
	\int_\Om g_k^Q\,d\mu
	&=  \int_\Om \left(2k\sum_{l}L_{k,l}\ch_{2B_{k,l}}\right)^Q\,d\mu\\
	&\le C_M^Q (2k)^Q \sum_{l}L_{k,l}^Q\mu(2 B_{k,l})\\
	&\le  2^{2Q} C_M^Q C_{\Om}\sum_{l}L_{k,l}^{Q}\quad\textrm{by }\eqref{eq:measure lower and upper bound 3}\\
	&\le  2^{2Q} C_M^Q  C_{\Om} j^Q \sum_{l}l_{k,l}^{Q}\quad\textrm{by }\eqref{eq:finiteness j condition}\\
	&\le  2^{2Q} C_M^Q   C_{\Om}^2 j^Q\sum_{l}\nu(B(f(y_{j,l}),l_{k,l}))
	\quad\textrm{by }\eqref{eq:measure lower and upper bound 3}\\
	&\le  2^{2Q} C_M^{Q+1}  C_{\Om}^2 j^Q\nu(V)
	\quad\textrm{by }\eqref{eq:finite nu measure estimate}\textrm{ (the same estimate applies)}\\
	&<\infty\quad\textrm{by }\eqref{eq:finite image assumption 3};
\end{align*}
recall that we are keeping $j$ fixed.
By reflexivity of the space $L^Q(\Om)$, we find a subsequence (not relabeled) and
$g\in L^Q(\Om)$
such that $g_k\to g$ weakly in $L^Q(\Om)$.
By Mazur's and Fuglede's lemmas
(Theorem \ref{thm:Mazur lemma} and Lemma \ref{lem:Fuglede lemma}),
we find convex combinations $\sum_{l=k}^{N_k}a_{k,l}g_l$ converging to $g$ in $L^Q(\Om)$,
and such that for
$Q$-a.e. curve $\gamma'$ in $\Om$ we have
\begin{equation}\label{eq:gamma prime ineq}
	\begin{split}
		\int_{\gamma'}g\,ds
		=\liminf_{k\to\infty}\int_{\gamma'}\sum_{l=k}^{N_k}a_{k,l}g_l\,ds
		&\ge \mathcal H_{\infty}^{1}(f(\gamma'\cap D^j))
		\quad \textrm{by }\eqref{eq:upper gradient property proved}.
	\end{split}
\end{equation}
By the properties of modulus,
see e.g. \cite[Lemma 1.34]{BB},
for $Q$-a.e. curve $\gamma$ in $\Om$ we have that \eqref{eq:gamma prime ineq} holds
for every subcurve $\gamma'$ of $\gamma$.
For $Q$-a.e. curve $\gamma$ in $\Om$, we also have that
$\int_{\gamma}g\,ds<\infty$ (see e.g. \cite[Proposition 1.37]{BB}).
Fix a curve $\gamma$ satisfying these two conditions.
We can write any open $U\subset (0,\ell_{\gamma})$ as an at most countable union of pairwise disjoint intervals
$U=\bigcup_{l}(a_l,b_l)$, and then
\begin{equation}\label{eq:gamma A ineq}
	\begin{split}
		\int_{U}g(\gamma(s))\,ds
		=\sum_{l}\int_{(a_l,b_l)}g(\gamma(s))\,ds
		&\ge \sum_{l}\mathcal H_{\infty}^{1}(f(\gamma((a_l,b_l))\cap D^j))\quad\textrm{by }\eqref{eq:gamma prime ineq}\\
		&\ge \mathcal H_{\infty}^{1}(f(\gamma(U)\cap D^j))
	\end{split}
\end{equation}
by the subadditivity of $\mathcal H_{\infty}^1$.
Then for any Borel set $S\subset (0,\ell_{\gamma})$, we can take a small $\delta>0$
and find an open set $U$ such that
$S\subset U\subset (0,\ell_{\gamma})$ and
\begin{align*}
	\int_{S}g(\gamma(s))\,ds
	\ge \int_{U}g(\gamma(s))\,ds-\delta
	&\ge \mathcal H_{\infty}^{1}(f(\gamma(U)\cap D^j))-\delta\quad\textrm{by }\eqref{eq:gamma A ineq}\\
	&\ge \mathcal H_{\infty}^{1}(f(\gamma(S)\cap D^j))-\delta.
\end{align*}
Letting $\delta\to 0$, we get
\[
\int_{S}g(\gamma(s))\,ds
\ge \mathcal H_{\infty}^{1}(f(\gamma(S)\cap D^j)).
\]
Note that $\mathcal H^1_{\infty}$ and $\mathcal H^1$ have the same null sets (see e.g. \cite[Lemma 2.1]{EvGa}).
Thus, using also the Borel regularity of $\mathcal L^1$, we conclude that whenever $S\subset [0,\ell_{\gamma}]$
with $\mathcal L^1(S)=0$, we have
$\mathcal H^{1}(f(\gamma(S)\cap D^j))=0$.
Recall that $\Om=\bigcup_{j=1}^{\infty}D^j\cup N$,
and that $Q$-a.e. curve avoids $N$.
Thus for $Q$-a.e. curve $\gamma$ in $\Om$,
we have that if $S\subset [0,\ell_{\gamma}]$ with $\mathcal L^1(S)=0$, then
\[
	\mathcal H^{1}(f(\gamma(S)))
	 \le \sum_{j=1}^{\infty}\mathcal H^{1}(f(\gamma(S)\cap D^j))
	=0.
\]
By Lemma \ref{lem:discontinuity set}, there is an at most countable set $H\subset D$ such that
$f$ is continuous at every point $y\in D\setminus H$.
Note that $Q$-a.e. curve avoids $H$, see \cite[Corollary 5.3.11]{HKSTbook}.
In total, for $Q$-a.e. curve $\gamma$ in $\Om$, we know that
\begin{equation}\label{eq:cont and Lusin}
f\circ\gamma\ \textrm{ is continuous}\quad \textrm{and}
\quad \textrm{if }S\subset [0,\ell_{\gamma}] \textrm{ with }\mathcal L^1(S)=0,\textrm{ then }
\mathcal H^{1}(f(\gamma(S)))
=0.
\end{equation}
Fix a new $\delta>0$.
Recall that we define
\[
H_f(y)=\limsup_{s\to 0} H_f(y,s)=\limsup_{s\to 0} \frac{L_f(y,s)}{l_f(y,s)}.
\]
Note that $H_f(y)\ge 1$ for every $y\in\Om$, by the fact that the space $X$ is connected.
Thus for every $y\in \Om$, for all sufficiently small $s>0$ we have
	\[
		\frac{L_f(y,s)}{l_f(y,s)}=H_f(y,s)\le (1+\delta)H_f(y).
	\]
It follows that
\begin{equation}\label{eq:lip and lh}
\begin{split}
\lip_f(y)^Q
&= \liminf_{s\to 0}\left(\frac{L_f(y,s)}{s}\right)^Q\\
&= (1+\delta)^{Q}H_f(y)^Q\liminf_{s\to 0}\left(\frac{l_f(y,s)}{s}\right)^Q\\
&\le C_{\Om}^{2}(1+\delta)^{Q}H_f(y)^Q \liminf_{s\to 0} \frac{\nu(B(f(y),l_f(y,s)))}{\mu(B(y,s))}
\quad\textrm{by }\eqref{eq:measure lower and upper bound 3}\\
&\le  C_{\Om}^{2}(1+\delta)^{Q}\Vert H_f\Vert_{L^{\infty}(\Om)}^Q
\liminf_{s\to 0} \frac{\nu(B(f(y),l_f(y,s)))}{\mu(B(y,s))}
\end{split}
\end{equation}
for $\mu$-a.e. $y\in\Om$.
For each $j\in\N$, consider any $\mu$-measurable function $h_j$ on $\Om$ that is at least
\[
 \Om\ni y\mapsto \frac{\nu(B(f(y),l_f(y,1/j)))}{\mu(B(y,1/j))}.
\]
By Lemma \ref{lem:Borel for lip f}, $\lip_f$ is $\mu$-measurable on $\Om$.
We get
\begin{equation}\label{eq:lip f Q}
\begin{split}
\int_{\Om}\lip_f^Q\,d\mu
&\le C_{\Om}^{2}(1+\delta)^{Q}\Vert H_f\Vert_{L^{\infty}(\Om)}^Q\int_{\Om}\liminf_{j\to \infty} h_j(y)\,d\mu(y)
\quad\textrm{by }\eqref{eq:lip and lh}\\
&\le C_{\Om}^{2}(1+\delta)^{Q}\Vert H_f\Vert_{L^{\infty}(\Om)}^Q \liminf_{j\to\infty} \int_{\Om}h_j(y)\,d\mu(y)
\quad\textrm{by Fatou}.
\end{split}
\end{equation}
Consider $j\in\N$ with $1/j< \beta$ (recall \eqref{eq:finite image assumption 3}), and
take balls $\{B(y_l,1/j)\}_{l}$, with $y_l\in \Om$,
which cover $\Om$, and such that the balls $B(y_l,2/j)$ can be divided into at most
$L\le C_d^4$ collections of pairwise disjoint balls
$\{B(y_l,2/j)\}_{l\in I_1},\ldots,\{B(y_l,2/j)\}_{l\in I_L}$.
Note that
\begin{align*}
\frac{\nu(B(f(y),l_f(y,1/j)))}{\mu(B(y,1/j))}
&\le  \sum_{l}\ch_{B(y_l,1/j)}(y)
\frac{\nu
	\big(\bigcup_{z\in B(y_l,1/j)\cap\Om} B(f(z),l_f(z,1/j))\big)}{\mu(B(y,1/j))}\\
&\le C_d^2 \sum_{l}\ch_{B(y_l,1/j)}(y)
\frac{\nu
	\big(\bigcup_{z\in B(y_l,1/j)\cap\Om} B(f(z),l_f(z,1/j))\big)}{\mu(B(y_l,1/j))},
\end{align*}
which is clearly $\mu$-measurable, and so $h_j$ can be taken to be the right-hand side. Now, if $z\in B(y_l,1/j)\cap\Om$,
and $z'\in B(y_{l'},1/j)\cap\Om$, where $B(y_l,2/j)\cap B(y_{l'},2/j)=\emptyset$, then
from the definition of $l_f(\cdot,\cdot)$ and from the injectivity of $f$ it follows 
(just as in \eqref{eq:injectivity for balls}) that
\begin{equation}\label{eq:injectivity for balls 2}
	B(f(z),l_f(z,1/j))\cap B(f(z'),l_f(z',1/j)) = \emptyset.
\end{equation}
We get
\begin{align*}
 \int_{ \Om} h_j(y)\,d\mu(y)
	& =  C_d^2 \int_{\Om} \sum_{l}\ch_{B(y_l,1/j)}(y)
	\frac{\nu
	\left(\bigcup_{z\in B(y_l,1/j)\cap\Om} B(f(z),l_f(z,1/j))\right)}{\mu(B(y_l,1/j))}\,d\mu(y)\\
	&\le C_d^2 \sum_{l} \nu\left(\bigcup_{z\in B(y_l,1/j)\cap\Om} B(f(z),l_f(z,1/j))\right)\\
	&= C_d^2 \sum_{m=1}^L\sum_{l\in I_m} \nu\left(\bigcup_{z\in B(y_l,1/j)\cap\Om} B(f(z),l_f(z,1/j))\right)\\
	&\le C_d^2 \sum_{m=1}^L \nu\left(\bigcup_{z\in \Om} B(f(z),l_f(z,1/j))\right)
	\quad\textrm{by }\eqref{eq:injectivity for balls 2}\\
	&\le C_d^6 \nu\left(\bigcup_{y\in \Om} B(f(y),l_f(y,1/j))\right)\\
	&= C_d^6\nu(V)<\infty
	\quad\textrm{by }\eqref{eq:finite image assumption 3}.
\end{align*}
Combining this with \eqref{eq:lip f Q}, we get
\begin{align*}
\Vert \lip_f\Vert_{L^Q(\Om)}^Q
&\le C_{\Om}^2 (1+\delta)^Q C_d^6 \Vert H_f\Vert_{L^{\infty}(\Om)}^Q \nu(V)\\
&\le C_{\Om}^2 (1+\delta)^Q C_d^6 \nu(V)\Vert a\Vert_{L^{\infty}(\Om)}^{Q-1}
\Vert H_f^{a,M}\Vert_{L^{\infty}(\Om)}^Q
\quad\textrm{by }\eqref{eq:L infinity comparison}.
\end{align*}
Letting $\delta\to 0$, we conclude
\begin{equation}\label{eq:lipf infinity}
	\Vert \lip_f\Vert_{L^Q(\Om)}^Q\le C_{\Om}^2 C_d^6 \nu(V)\Vert a\Vert_{L^{\infty}(\Om)}^{Q-1}
	\Vert H_f^{a,M}\Vert_{L^{\infty}(\Om)}^Q.
\end{equation}	
Thus $\lip_f<\infty$ $\mu$-a.e. in $\Om$.
Using the function $\infty\ch_F$ where $F\supset \{\lip_f=\infty\}$
is a $\mu$-negligible Borel set, we see that
for $Q$-a.e. curve $\gamma$ in $\Om$ we have,
with $T:=[0,\ell_{\gamma}]\cap \gamma^{-1}(\{\lip_f<\infty\})$,
that $\mathcal L^1([0,\ell_{\gamma}]\setminus T)=0$.
Moreover, $f\circ\gamma$ is continuous by \eqref{eq:cont and Lusin}, and thus
\begin{align*}
d_Y(f(\gamma(0)),f(\gamma(\ell_{\gamma})))
&\le \mathcal H^1(f(\gamma([0,\ell_{\gamma}])))\\
&\le \mathcal H^1(f(\gamma(T)))\quad\textrm{by }\eqref{eq:cont and Lusin}\textrm{ (second part)}\\
&\le \int_{T}\lip_f(\gamma(t))\,dt\quad\textrm{by Lemma \ref{lem:abs cont}}\\
&\le \int_{\gamma}\lip_f\,dt.
\end{align*}
Hence $\lip_f\in L^Q(\Om)$ is a $Q$-weak upper gradient of $f$ in $\Om$.
In conclusion, $f \in D^Q(\Om;Y)$. Moreover, by \eqref{eq:lipf infinity} we have
\[
\Vert f\Vert_{D^Q(\Om;Y)}^Q \le \Vert \lip_f\Vert_{L^Q(\Om)}^Q
\le C_d^6 C_{\Om}^2 \nu(V)\Vert a\Vert_{L^{\infty}(\Om)}^{Q-1}
\Vert H_f^{a,M}\Vert_{L^{\infty}(\Om)}^Q,
\]
proving \eqref{eq:C3 energy estimate Q} with $C_4=C_d^6 C_{\Om}^2$.
\end{proof}

In Part 2 of the proof, the idea of first proving absolute continuity on curves,
and then examining $\lip_f$, comes from Williams \cite{Wi}.
This part of the proof was moreover quite similar to \cite[Section 5]{LaZh2}, but
as usual, extra difficulties arose from the fact that $f$ is not assumed to be a homeomorphism,
merely injective.

\section{Corollaries and examples}\label{sec:cor}

Since the formulations of Theorems \ref{thm:main theorem}, \ref{thm:main theorem 2}, and
\ref{thm:main theorem 3} are rather long and technical, we will consider three
corollaries, as well as a few examples.

\begin{corollary}\label{cor:main lip}
	Suppose $\Om\subset X$ is open and bounded,
	$f\colon \Om\to Y$ is bounded, $M\ge 1$,
	$\kappa\ge \mu$ is a finite Radon measure on $\Om$,
	$\Lip_f^{\kappa,M}(x)\le 1$ for $\widetilde{\mathcal{H}}^1$-a.e. $x\in\Om$,
	and there is $0<C_{\Om}<\infty$ such that
	\[
	\frac{\mu(B(x,r))}{r^Q}\le C_{\Om}\ \textrm{ for all }x\in \Om\ \textrm{and}\ r>0.
	\]
	Then  $f\in \BV(\Om;Y)$.
\end{corollary}
\begin{proof}
This is obtained by applying Theorem \ref{thm:main theorem} with
the choices: $N$ is the subset of $\Om$ where $\Lip_f^{\kappa,M}>1$,
$A=\Om\setminus N$, $D=W=\emptyset$, $h=1$, and $\eps=1$.
Note that $\Mod_1(\Gamma_N)=0$ by Lemma \ref{lem:curve intersects set}, and that
$f\in L^{\infty}(\Om;Y)\subset L^1(\Om;Y)$.
\end{proof}

Many BV mappings, which may be highly discontinuous, can be seen as ``generalized Lipschitz''
in the sense of $\Lip_f^{\kappa,M}$ being bounded. Consider the following example.

\begin{example}\label{ex:lipschitz}
	Let the spaces $(X,d,\mu)$ and $(Y,d_Y,\nu)$ be $\R^2$ and $\R^1$, respectively,
	equipped with the Euclidean metrics and the Lebesgue measures $\mathcal L^2$ and $\mathcal L^1$.
	Let $\Om=(0,1)\times (0,1)$, and for $(x_1,x_2)\in \Om$ let
	\[
	f(x_1,x_2):=\lambda((0,x_2])
	\]
	for a finite Radon measure $\lambda\ge\mathcal L^1$.
	If $\lambda$ is not absolutely continuous with respect to $\mathcal L^1$, then $f$ is
	obviously not Lipschitz, and if $\lambda$ has Dirac masses, $f$ is not even continuous.
	Thus the existing literature
	results written for continuous functions, as described in the Introduction, are not applicable.
	On the other hand, define the Radon measure
	$\kappa:=\mathcal L^1\times \lambda$ on $\Om$. For every $x=(x_1,x_2)\in \Om$, we get
	\begin{equation}\label{eq:Lip 2 estimate}
	\begin{split}
		 \Lip_f^{\kappa,2}(x)
		&=\limsup_{r\to 0}\frac{L_f(x,r)}{r}\frac{\mathcal L^2(B(x,2r))}{\kappa(B(x,2r))}\\
		&\le  \limsup_{r\to 0}\frac{\lambda([x_2-r,x_2+r])}{r}
		\frac{\mathcal L^2(B(x,2r))}{\kappa(B(x,2r))}\\
		&\le \limsup_{r\to 0}\frac{\lambda([x_2-r,x_2+r])}{r}
		\frac{4\pi r^2}{2r\lambda ([x_2-r,x_2+r])}\\
		&\le 2\pi.
	\end{split}
	\end{equation}
	Now Corollary \ref{cor:main lip} gives $f\in \BV(\Om;\R)$.
\end{example}

Note that in \eqref{eq:Lip 2 estimate},
the factor ``$2$'' in $\kappa(B(x,2r))$ is crucial for obtaining
the quantity $\lambda([x_2-r,x_2+r])$ in the denominator on the third line.
This demonstrates that the number $M$ in the definition of $\Lip_f^{\kappa,M}$
provides crucial flexibility.

In the next corollary, the idea is that we choose $a=1$ and then the generalized Lipschitz number reduces to the
ordinary one.

\begin{corollary}\label{cor:literature}
	Suppose $\Om\subset X$ is open and bounded, $1\le p<\infty$,
	$f\colon \Om\to Y$ is bounded,
	$\Lip_f(x)<\infty$ for $\widetilde{\mathcal H}^p$-a.e. $x\in\Om$,
	$\Vert \Lip_f\Vert_{L^p(\Om)}<\infty$,
	and there is $0<C_{\Om}<\infty$ such that
	\[
	\frac{\mu(B(x,r))}{r^Q}\le C_{\Om}\ \textrm{ for all }x\in \Om\ \textrm{and}\ r>0.
	\]
	Then  $f\in N^{1,p}(\Om;Y)$.
\end{corollary}
\begin{proof}
	This is obtained by applying Theorem \ref{thm:main theorem 2} with
	the choices: $M=1$, $a=1$, $N$ is the subset of $\Om$ where $\Lip_f=\infty$,
	$A=\Om\setminus N$, $D=W=\emptyset$, and $h=\Lip_f$.
	Note that $\Mod_p(\Gamma_N)=0$ by Lemma \ref{lem:curve intersects set}, and that
	$f\in L^{\infty}(\Om;Y)\subset L^p(\Om;Y)$.
\end{proof}

In this way, we recovered a result very similar to the known results in the literature,
see e.g. \cite[Theorem 1.5]{BC} and recall the results described in the Introduction.
However, the main advantage of the generalized Lipschitz and distortion numbers becomes apparent when choosing suitably
large measures $\kappa$ and functions $a$, which force these numbers to be
bounded. This was the idea in Corollary \ref{cor:main lip}, and also
in the next corollary.

\begin{corollary}\label{cor:main theorem}
	Suppose $\Om\subset X$ is open and bounded, $f\colon \Om\to Y$ is injective and bounded,
	$M\ge 1$, $1\le p\le Q$, 
	there is $a\colon \Om\to [1,\infty)$ in $L^{p^*(Q-1)/Q}(\Om)$ such that
		$H_f^{a,M}\le 1$ for $\widetilde{\mathcal H}^p$-a.e. $x\in\Om$,
		and there is $0<C_{\Om}<\infty$ such that
		\[
		\frac{\mu(B(x,r))}{r^Q}\le C_{\Om}
		\quad\textrm{and}\quad
		\frac{\nu(B(f(x),r))}{r^Q}\ge C_{\Om}^{-1}
		\ \textrm{ for all }x\in \Om.
		\]
		 Then
		$f\in N^{1,p}(\Om;Y)$.
\end{corollary}
\begin{proof}
	This is obtained by applying Theorem \ref{thm:main theorem 3}
	with the choices: $N$ is the subset of $\Om$ where
	$H_f^{a,M}>1$,
	$D=\Om\setminus N$,  and $h=1$.
	Note that $\Mod_p(\Gamma_N)=0$ by Lemma \ref{lem:curve intersects set}.
	Moreover, as long as $\Om$ consists of at least two points,
	we can choose $x_0\in \Om$ and a number $0<\beta<\diam \Om$, and then
	\begin{align*}
		\nu\left(\bigcup_{y\in D}B(f(y),l_f(y,\beta))\right)
		&\le \nu\left(\bigcup_{y\in \Om}B(f(y),\diam f(\Om))\right)\\
		&\le \nu(B(f(x_0),2\diam f(\Om))),
	\end{align*}
	which is finite by our standing assumption that balls have finite measure.
	Thus \eqref{eq:finite image assumption 3} is also fulfilled.
\end{proof}

Recall that Theorem \ref{thm:main intro} from the Introduction states that:

\begin{itemize}
	\item Suppose $(X,d,\mu)$ and $(Y,d_Y,\nu)$ are Ahlfors $Q$-regular, and $1\le p\le Q$.
	Let $f\colon X\to Y$ be injective and bounded, and $M\ge 1$.
	Then:
	\begin{enumerate}
		\item If there exists a Radon measure $\kappa\ge \mu$ such that
		$\Lip_f^{\kappa,M}(x)\le 1$ for $\widetilde{\mathcal{H}}^1$-a.e. $x\in X$,
		then $f\in \BV_{\loc}(X;Y)$;
		\item If  there exists a function $a\colon X\to [1,\infty)$ belonging to
		$L_{\loc}^{p^*(Q-1)/Q}(X)$ such that
		$H_f^{a,M}(x)\le 1$ for $\widetilde{\mathcal H}^p$-a.e. $x\in X$, then
		$f\in N_{\loc}^{1,p}(X;Y)$.
	\end{enumerate}
\end{itemize}

\begin{proof}[Proof of Theorem \ref{thm:main intro}]
(1):
For every $x\in X$ we find $r>0$ such that
$\kappa(B(x,r))<\infty$, and then from Corollary \ref{cor:main lip} we get $f\in \BV(B(x,r);Y)$.
In conclusion, $f\in \BV_{\loc}(X;Y)$.

(2):
For every $x\in X$ we find $r>0$ such that
$\Vert a\Vert_{L^{p^*(Q-1)/Q}(B(x,r))}<\infty$,
and then from Corollary \ref{cor:main theorem} we get $f\in N^{1,p}(B(x,r);Y)$.
In conclusion, $f\in N^{1,p}_{\loc}(X;Y)$.
\end{proof}

\begin{example}\label{ex:quasiconformal}
	Let $(X,d,\mu)$ and $(Y,d_Y,\nu)$ both be
	$\R^2$ equipped with the Euclidean metric and the Lebesgue measure $\mathcal L^2$.
	Let $\Om=(0,1)\times (0,1)$, and for $(x_1,x_2)\in \Om$ let
	\[
	f((x_1,x_2)):=(f_1(x_1),f_2(x_2)),
	\]
	where
	\[
	f_1(x_1):=x_1\quad\textrm{and}\quad f_2(x_2):=\int_0^{x_2}g(s)\,ds
	\]
	for a function $g\in L^2((0,1))$, $g\ge 1$.
	For every $x=(x_1,x_2)\in \Om$, we have
	\begin{equation}\label{eq:Hf is comparable to limsup}
	\limsup_{r\to 0}\vint{B(x_2,r)}g(s)\,ds
	\le H_f(x)\le 2\limsup_{r\to 0}\vint{B(x_2,r)}g(s)\,ds.
	\end{equation}
	Thus $H_f=\infty $ in a set $(0,1)\times D$, where $D\subset (0,1)$ may clearly be of dimension strictly
	greater than $0$, and then $E:=(0,1)\times D$ is of dimension strictly greater than $1$.
	Thus $E$ is not of $\sigma$-finite $\mathcal H^1$-measure, and the results in the literature,
	as described in the Introduction, are not applicable.
	(Note that $\mathcal H^1$ and $\widetilde{\mathcal H}^1$ are now comparable.)
	On the other hand, let $\widehat{g}(x_1,x_2):=g(x_2)$ and $a:= \widehat{g}^2$,
	so that $a\in L^1(\Om)$. For every $x=(x_1,x_2)\in \Om$, we get
	\begin{align*}
		 H_f^{a,2}(x)
		&= \limsup_{r\to 0}\frac{L_f(x,r)}{l_f(x,r)}
		\left(\vint{B(x,2r)}a\,d\mathcal L^2\right)^{(1-2)/2}\\
		&\le 2\limsup_{r\to 0}\vint{(x_2-r,x_2+r)}g\,ds
		\left(\vint{B(x,2r)}a\,d\mathcal L^2\right)^{-1/2}\\
		&\le 2\limsup_{r\to 0}\vint{(x_2-r,x_2+r)}g\,ds
		\left(\vint{B(x,2r)}\widehat{g}\,d\mathcal L^2\right)^{-1}\quad\textrm{by H\"older}\\
		&\le 2\limsup_{r\to 0}\vint{(x_2-r,x_2+r)}g\,ds
		\left(\frac{1}{4\pi r^2}\int_{(x_2-r,x_2+r)}\int_{(x_1-r,x_1+r)}g(s)\,dt\, ds\right)^{-1}\\
		&= 2\pi.
	\end{align*}
Now Corollary \ref{cor:main theorem} gives $f\in N^{1,1}(\Om;\R^2)$.
\end{example}

In the literature, it is known that $E=\{H_f=\infty\}$ cannot in general be allowed to be larger than $\sigma$-finite with respect
to $\widetilde{\mathcal H}^1$, see e.g. \cite[Remark 1.2(b)]{KoRo} and \cite[Remark 1.9]{Wi}.
However, Example \ref{ex:quasiconformal} demonstrates that we can sometimes get around this limitation
by using the generalized distortion number $H_f^{a,M}$.
A similar example was already given previously
in \cite[Example 5.27]{LahGFQ}, but there we were limited 
to being able to handle functions $g$ with a suitable structure, whereas
above, $1\le g\in L^2((0,1))$ can be arbitrary.
This flexibility is again achieved by means of the factor $M$ in $H_f^{a,M}$;
in this case $M=2$ sufficed.

On the other hand, note that clearly the mapping $f$ in Example \ref{ex:quasiconformal} is actually in $N^{1,2}(\Om;\R^2)$
and not only in $N^{1,1}(\Om;\R^2)$. This raises the question of whether there could be a sharper version of
Corollary \ref{cor:main theorem}.

Finally we return to the question of the connection between the exceptional set $E$ 
and the approximate jump set of a BV function.
Our reasoning will be similar to \cite[Corollary 5.22]{LahGFQ}, but in the theory developed in that paper,
$f$ was always assumed to be continuous, which is of course not natural from the viewpoint of BV theory.

\begin{example}\label{ex:set E}
	Suppose $f\colon \R^n\to \R^n$ is a homeomorphism, with $n\ge 2$, and
	there exists a set $E$ of $\sigma$-finite $\mathcal H^{n-1}$-measure such that
	$H_f(x)<\infty$ for every $x\in\R^n\setminus E$.
	Assume also that $H_f\in L_{\loc}^{n/(n-1)}(\R^n)$.

	Additionally, assume that $E$ can be presented as a
	union $E=\bigcup_{j=0}^{\infty}E_j$, where each $E_j$ is a Borel set with
	$\mathcal H^{n-1}(E_0)=0$, $0<\mathcal H^{n-1}(E_j)<\infty$ for all $j\in\N$, and
	\begin{equation}\label{eq:lower density condition}
	\liminf_{r\to 0}\frac{\mathcal H^{n-1}(B(x,r)\cap E_j)}{r^{n-1}}>0
	\end{equation}
	for $\mathcal H^{n-1}$-a.e. $x\in E_j$, for each $j\in\N$.
	This is a regularity assumption that is satisfied for example if $E$ is an $n-1$-rectifiable set
	(see e.g. \cite[Theorem 2.83]{AFP}).

	Define $\kappa$ by
	\begin{equation}\label{eq:choice of kappa}
	\kappa :=\mathcal L^n +\sum_{j=1}^{\infty} 2^{-j}\mathcal H^{n-1}(E_j)^{-1}\mathcal H^{n-1}\mres_{E_j}.
	\end{equation}
	At points $x\in E$ where \eqref{eq:lower density condition} is satisfied, by the continuity of $f$ we get
	\[
	\Lip_f^{\kappa,1}(x)=
	\limsup_{r\to 0}\frac{L_f(x,r)}{r}\frac{\mathcal L^n(B(x,r))}{\kappa(B(x,r))}=0.
	\]
	Thus we obtain $\Lip_f^{\kappa,1}(x)=0$ for $\mathcal H^{n-1}$-a.e. $x\in E$;
	this gives an exceptional (Borel) $N\subset E$ with $\mathcal H^{n-1}(N)=0$. Meanwhile,
	$H_f^{\kappa,1}(x)\le H_f(x)<\infty$ for
	every $x\in \R^n\setminus E$.
	Now, for a bounded open set $\Om\subset \R^n$, we apply the first part of Theorem \ref{thm:main theorem}
	with the choices $A=\Om\cap E\setminus N$, $D=\Om\setminus E$, and
	\[
	h=\Lip_f^{\kappa,1}\ch_{A}+(H_f^{\kappa,1})^{n/(n-1)}\ch_D;
	\]
	measurability is now straightforward to show, since $f$ is continuous.
	Note also that since $f$ is a homeomorphism, it is locally bounded and then
	\eqref{eq:finite image assumption} is satisfied with every choice
	of $\beta>0$.
	For any $0<\eps\le 1$, we obtain
	\begin{equation}\label{eq:Df estimate}
	\begin{split}
	\Vert Df\Vert(\Om)
	&\le
	C_1\eps^{-1/(n-1)}\left(\int_{A}\Lip_f^{\kappa,1}\,d\kappa+ \int_{D}(H_f^{\kappa,1})^{n/(n-1)}\,d\kappa\right)
	+C_1\eps(\kappa(\Om)+\mathcal L^n(V))\\
	&\le C_1\left(\eps^{-1/(n-1)}\int_{\Om}H_f^{n/(n-1)}\,d\mathcal L^n
	+\eps \kappa(\Om)+\eps\mathcal L^n(V)\right).
	\end{split}
	\end{equation}
Thus $f\in \BV_{\loc}(\R^n;\R^n)$; for this, we can simply choose $\eps=1$.
On the other hand, using the fact that we can choose arbitrarily small $\eps>0$, we also see that
$\Vert Df\Vert$ is absolutely continuous with respect to $\mathcal L^n$.
In the Euclidean BV theory (see e.g. \cite[Section 3]{AFP}), this implies that in fact
$f\in W_{\loc}^{1,1}(\R^n;\R^n)$.

The measure  $\kappa$ given in \eqref{eq:choice of kappa} charges the set $E$,
which is $\sigma$-finite with respect to $\mathcal H^{n-1}$, just like the
approximate jump set $J_f$ of a BV function;
recall \eqref{eq:jump value 1}.
Thus the first line of \eqref{eq:Df estimate} appears to give an upper bound for the jump part
of the variation measure,
$\Vert Df\Vert^j=|f^+-f^-|\, d\mathcal H^{n-1}\mres{J_f}$ (recall \eqref{eq:jump part representation}).
In conclusion, the exceptional set $E$ can be interpreted to correspond
to the approximate jump set $J_f$ of the BV function $f$.
Since $f$ is now continuous by assumption, however, 
the approximate jump set is empty and so we actually obtained $f\in W_{\loc}^{1,1}(\R^n;\R^n)$.
\end{example}

\end{document}